%% file: ergodic_gradient_descent.tex
\ifdefined\siam
\documentclass[final]{siamltex}
\usepackage{gd_siam}
\else
\documentclass[11pt]{article}
\usepackage{gd}
\usepackage{fullpage}
\fi

\def\indset{\ensuremath{\mc{I}}}
\def\event#1{\ensuremath{E_{#1}}}

\title{Ergodic Mirror Descent}
\author{John C.\ Duchi\footnotemark[1] \and
  Alekh Agarwal\footnotemark[1] \and
  Mikael Johansson\footnotemark[2] \and
  Michael I.\ Jordan\footnotemark[1]\ \footnotemark[3]}

\newcommand{\comment}[1]{}

\begin{document}

\maketitle

\renewcommand{\thefootnote}{\fnsymbol{footnote}}

\footnotetext[1]{Department of Electrical Engineering and Computer Sciences,
  University of California, Berkeley; Berkeley, CA USA. Email:
  \texttt{\{jduchi,alekh,jordan\}@eecs.berkeley.edu}.  JCD was supported by an
  NDSEG fellowship, and AA was supported by a Microsoft Research Fellowship
  and a Google PhD Fellowship. Also supported in part by the U.S.\ Army
  Research Laboratory and the U.S.\ Army Research Office under contract/grant
  number W911NF-11-1-0391.}

\footnotetext[2]{School of Electrical Engineering, Royal Institute of
  Technology (KTH); Stockholm, Sweden. Email:
  \texttt{mikael.johansson@ee.kth.se}}

\footnotetext[3]{Department of Statistics, University of California, Berkeley;
  Berkeley, CA USA.}

\renewcommand{\thefootnote}{\arabic{footnote}}

\begin{abstract}
  We generalize stochastic subgradient descent methods to situations in which
  we do not receive independent samples from the distribution over which we
  optimize, instead receiving samples coupled over time. We show
  that as long as the source of randomness is suitably ergodic---it converges
  quickly enough to a stationary distribution---the method enjoys strong
  convergence guarantees, both in expectation and with high probability. This
  result has implications for stochastic optimization in high-dimensional
  spaces, peer-to-peer distributed optimization schemes, decision problems
  with dependent data, and stochastic optimization problems over combinatorial
  spaces.
\end{abstract}

\input{introduction}

\input{assumptions}

\input{main-results}

\input{examples}

\input{experiments}

\input{analysis}

\input{conclusions}

\subsection*{Acknowledgments}

We thank Lester Mackey for several interesting questions he posed that helped
lead to this work. In addition, we thank the three anonymous reviewers and the
editor for many insightful comments and suggestions. 

\appendix

\input{appendix}

\ifdefined\siam
\bibliographystyle{siam}
\bibliography{bib}
\else
\bibliographystyle{abbrv}
{\small
\bibliography{bib}
}
\fi

\end{document}

%% file: introduction.tex
\section{Introduction}

In this paper, we analyze a new algorithm, Ergodic Mirror Descent, for
solving a class of stochastic optimization problems. We begin with a
statement of the problem.  Let $\{F(\cdot; \statsample), \statsample
\in \statsamplespace\}$ be a collection of closed convex functions
whose domains contain the common closed convex set $\xdomain \subseteq
\R^d$. Let $\stationary$ be a probability distribution over the
statistical sample space $\statsamplespace$ and consider the convex
function $f: \xdomain \rightarrow \R$ defined by the expectation
\begin{equation}
  \label{eqn:f-def}
  f(x) \defeq \E_\stationary[F(x; \statsample)]
  = \int_\statsamplespace F(x; \statsample) d\stationary(\statsample).
\end{equation}
We study algorithms for solving the following problem:
\begin{equation}
  \minimize_x ~ f(x)
  ~~~ \mbox{subject to} ~~~
  x \in \xdomain.
  \label{eqn:objective}
\end{equation}

A wide variety of stochastic optimization methods for solving the
problem~\eqref{eqn:objective} have been explored in an extensive
literature~\cite{RobbinsMo51, PolyakJu92, NedicBe01, KushnerYi03,
  NemirovskiJuLaSh09}. We study procedures that do not assume it is
possible to receive samples from the distribution $\stationary$,
instead receiving samples $\statsample$ from a stochastic process
$\statprob$ indexed by time $t$, where the stochastic process
$\statprob$ converges to the stationary distribution
$\stationary$. This is a natural relaxation, because in many
circumstances the distribution $\stationary$ is not even known---for
example in statistical applications---and we cannot receive
independent samples. In other scenarios, it may be hard to even draw
samples from $\stationary$ efficiently, such as when
$\statsamplespace$ is a high-dimensional or combinatorial space, but
it is possible~\cite{JerrumSi96} to design Markov chains that converge
to the distribution $\stationary$. Further, in computational
applications, it is often unrealistic to assume that one actually has
access to a source of independent randomness, so studying the effect
of correlation is natural and important~\cite{ImpagliazzoZu89}.

Our approach to solving the problem~\eqref{eqn:objective} is related to
classical stochastic gradient descent
algorithms~\cite{RobbinsMo51,PolyakJu92}, where one assumes access to samples
$\statsample$ from the distribution $\stationary$ and performs gradient
updates using $\nabla F(x; \statsample)$.
When $\stationary$ is concentrated on a set of $n$ points and the functions
$F$ are not necessarily differentiable, the incremental subgradient
method of Nedi\'{c} and Bertsekas~\cite{NedicBe01} applies, and the objective
is of the form $f(x) = \ninv \sum_{i=1}^n f_i(x)$. More generally, our problem
belongs to the family of stochastic problems with exogenous correlated
noise~\cite{KushnerYi03} where the goal is to minimize $\E_\stationary[F(x;
  \statsample)]$ as in the objective~\eqref{eqn:objective}, but we have access
only to samples $\statsample$ that are not independent over time. Certainly a
number of researchers in control, optimization, stochastic approximation, and
statistics have studied settings where stochastic data is not i.i.d.\ (see,
for example, the books~\cite{KushnerYi03,Spall03} and the numerous references
therein).  Nonetheless, classical results in this setting are asymptotic in
nature and generally do not provide finite sample or high-probability
convergence guarantees; our work provides such results.

Our method borrows from standard stochastic subgradient and stochastic mirror
descent methodology~\cite{NemirovskiYu83,NemirovskiJuLaSh09}, but we
generalize this work in that we receive samples not from the distribution
$\stationary$ but from an ergodic process $\statsample_1, \statsample_2,
\ldots$ converging to the stationary distribution $\stationary$. In spite of
the new setting, we do not modify standard stochastic subgradient algorithms;
our algorithm receives samples $\statsample_t$ and takes mirror descent steps
with respect to the subgradients of $F(x; \statsample_t)$.  Consequently, our
approach generalizes several recent works on stochastic and non-stochastic
optimization, including the randomized incremental subgradient
method~\cite{NedicBe01} as well as the Markov incremental subgradient
method~\cite{JohanssonRaJo09,RamNeVe09a}. There are a number of
applications of this work: in control problems, data is often coupled over
time or may come from an autoregressive process~\cite{KushnerYi03}; in
distributed sensor networks~\cite{LesserOrTa03}, a set of wireless sensors
attempt to minimize an objective corresponding to a sequence of correlated
measurements; and in statistical problems, data comes from an unknown
distribution and may be dependent~\cite{Yu94}.  See our examples and
experiments in \S~\ref{sec:examples} and \S~\ref{sec:experiments}, as well as
the examples in the paper by Ram et al.~\cite{RamNeVe09a}, for other
motivating applications.

The main result of this paper is that performing stochastic gradient
or mirror descent steps as described in the previous paragraph is a
provably convergent optimization procedure. The convergence is
governed by problem-dependent quantities (namely the radius of
$\xdomain$ and the Lipschitz constant of the functions $F$) familiar
from previous results on stochastic methods~\cite{NedicBe01,
  Zinkevich03, NemirovskiJuLaSh09} and also depends on the rate at
which the stochastic process $\statsample_1, \statsample_2, \ldots$
converges to its stationary distribution. Our three main convergence
theorems characterize the convergence rate of Ergodic Mirror Descent
in terms of the mixing time $\tmix$ (the time it takes the process
$\statsample_t$ to converge to the stationary distribution
$\stationary$, in a sense we make precise later) in expectation, with
high probability, and when the mixing times of the process are
themselves random. In particular, we show that this rate is
$\order\left(\sqrt{\frac{\tmix}{T}}\right)$ for a large class of
ergodic processes, both in expectation and with high probability.  We
also give a lower bound that shows that our results are tight: they
cannot (in general) be improved by more than numerical constants.

The remainder of the paper is organized as follows.
Section~\ref{sec:algorithm} contains our main assumptions and a description of
the algorithm.  Following that, we collect our main technical
results in \S~\ref{sec:main-results}.  We expand on these results in
example corollaries throughout \S~\ref{sec:examples} and give
numerical simulations exploring our algorithms in
\S~\ref{sec:experiments}.  We provide complete proofs of all our results
in \S~\ref{sec:analysis} and the appendices.

\paragraph{Notation}
For the reader's convenience, we collect our (standard) notation here.
A function $f$ is $\lipobj$-Lipschitz with respect to a norm $\norm{\cdot}$ if
$|f(x) - f(y)| \le \lipobj \norm{x - y}$. The dual norm $\dnorm{\cdot}$ to a
norm $\norm{\cdot}$ is defined by $\dnorm{z} \defeq \sup_{\norm{x} \le 1} \<z,
x\>$. A function $\prox$ is strongly convex with respect to the norm
$\norm{\cdot}$ over the domain $\xdomain$ if
\begin{equation*}
  \prox(y) \ge \prox(x) + \<\nabla \prox(x), y - x\> + \half \norm{x - y}^2
  ~~ \mbox{for} ~ x, y \in \xdomain.
\end{equation*}
For a convex function $f$, we let $\partial f(x) = \{g \in \R^d \mid f(y) \ge
f(x) + \<g, y - x\>\}$ denote its subdifferential. For a matrix $A \in \R^{n
  \times m}$, we let $\singval_i(A)$ denote its $i$th largest singular value,
and when $A \in \R^{n \times n}$ is symmetric we let $\lambda_i(A)$ denote its
$i$th largest eigenvalue. The all-ones vector is $\onevec$, and we denote the
transpose of the matrix $A$ by $A^\top$. We let $[n]$ denote the set $\{1,
\ldots, n\}$. For functions $f$ and $g$, we write $f(n) = \order(g(n))$ if
there exist $N < \infty$ and $C < \infty$ such that $f(n) \le C g(n)$ for $n
\ge N$, and $f(n) = \Omega(g(n))$ if there exist $N < \infty$ and $c > 0$ such
that $f(n) \ge c g(n)$ for $n \ge N$. For a probability measure $P$ and
measurable set or event $A$, $P(A)$ denotes the mass $P$ assigns
$A$.

%% file: assumptions.tex
\section{Assumptions and algorithm}
\label{sec:algorithm}

We now turn to describing our algorithm and the assumptions underlying
it.  We begin with a description of the algorithm, which is familiar
from the literature on mirror descent
algorithms~\cite{NemirovskiYu83,BeckTe03}.  Specifically, we
generalize the stochastic mirror descent
algorithm~\cite{NemirovskiYu83,NemirovskiJuLaSh09}, which in turn
generalizes gradient descent to elegantly address non-Euclidean
geometry.  The algorithm is based on a prox-function $\prox$, a
differentiable convex function defined on $\xdomain$ assumed
(w.l.o.g.\ by scaling) to be $1$-strongly convex with respect to the
norm $\norm{\cdot}$ over $\xdomain$.  The Bregman divergence
$\divergence$ generated by $\prox$ is defined as
\begin{equation}
  \divergence(x, y) \defeq
  \prox(x) - \prox(y) - \<\nabla \prox(y), x - y\>
  \ge \half \norm{x - y}^2.
  \label{eqn:bregman-convex}
\end{equation}
We assume $\xdomain$ is compact and
that there exists a radius $\radius < \infty$ such that
\begin{equation}
  \label{eqn:compactness}
  \divergence(x, y) \le \half \radius^2
  ~~~ \mbox{for}~ x, y \in \xdomain.
\end{equation}

The Ergodic Mirror Descent (EMD) algorithm is an
iterative algorithm that maintains a parameter $x(t) \in \xdomain$,
which it updates using stochastic gradient information to form $x(t +
1)$. Specifically, let $\statprob^t$ denote the distribution of the
stochastic process $\statprob$ at time $t$. We assume that we receive
a sample $\statsample_t \sim \statprob^t$ at each time step $t$.
Given $\statsample_t$, EMD computes the update
\begin{equation}
  g(t) \in \partial F(x(t); \statsample_t), ~~~ x(t + 1) = \argmin_{x
    \in \xdomain} \left\{ \<g(t), x\> + \frac{1}{\stepsize(t)}
  \divergence(x, x(t)) \right\}.
  \label{eqn:ergodic-md}
\end{equation}
The initial point $x(1)$ may be selected arbitrarily in $\xdomain$, and here
$\stepsize(t)$ is a non-increasing (time-dependent) stepsize.  The
algorithm~\eqref{eqn:ergodic-md} reduces to projected gradient descent with
the choice $\prox(x) = \half \ltwo{x}^2$, since then $\divergence(x, y) =
\half \ltwo{x - y}^2$.

Our main assumption on the functions $F(\cdot; \statsample)$ regards their
continuity and subdifferentiability properties, though we require a bit more
notation. Let $\gradfunc(x; \statsample) \in \partial F(x; \statsample)$
denote a fixed and measurable element of the subgradient of $F(\cdot;
\statsample)$ evaluated at the point $x$, where (without loss of generality) we
assume that in the EMD algorithm~\eqref{eqn:ergodic-md} we have $g(t) =
\gradfunc(x(t); \statsample_t)$. We let $\mc{F}_t$ denote the $\sigma$-field
of the first $t$ random samples $\statsample_1, \ldots, \statsample_t$ from
the stochastic process $\statprob$ (that is, $\statsample_t$ is drawn
according to $\statprob^t$).
We make one of the following two assumptions, where
in each the norm $\norm{\cdot}$ is the norm with respect to which
$\prox$ is strongly convex~\eqref{eqn:bregman-convex}:
\begin{assumption}[Finite single-step variance]
  \label{assumption:one-step-variance}
  Let $x$ be measurable with respect to the $\sigma$-field
  $\mc{F}_{t-1}$. There exists a constant $\lipobj < \infty$ such
  that with probability 1
  \begin{equation*}
    \E[\dnorm{\gradfunc(x; \statsample_t)}^2 \mid \mc{F}_{t-1}] \le
    \lipobj^2.
  \end{equation*}
\end{assumption}
\begin{assumption}
  \label{assumption:F-lipschitz}
  For $\stationary$-almost every $\statsample$,
  the functions $F(\cdot; \statsample)$ are $\lipobj$-Lipschitz
  continuous functions
  with respect to a norm $\norm{\cdot}$ over $\xdomain$. That is,
  \begin{equation*}
    |F(x; \statsample) - F(y; \statsample)| \le \lipobj \norm{x - y}
    ~~~ \mbox{for~} x, y \in \xdomain.
  \end{equation*}
\end{assumption}
As a consequence of Assumption~\ref{assumption:F-lipschitz}, for any
$g \in \partial F(x; \statsample)$ we have that $\dnorm{g} \le \lipobj$
(e.g.,~\cite{HiriartUrrutyLe96}), and it is clear that the expected function
$f$ is also $\lipobj$-Lipschitz. Assumption~\ref{assumption:F-lipschitz}
implies Assumption~\ref{assumption:one-step-variance}, though
Assumption~\ref{assumption:one-step-variance} still guarantees
$f$ is $\lipobj$-Lipschitz, and under either assumption, we have
\begin{equation}
  \label{eqn:gradient-variance}
  \E\left[\dnorm{\gradfunc(x; \statsample)}^2\right] = \E\left[
    \E\left[\dnorm{\gradfunc(x; \statsample)}^2 \mid
      \mc{F}_{t-1}\right] \right] \le \lipobj^2.
\end{equation}

Having described the family of functions $\{F(\cdot; \statsample) :
\statsample \in \statsamplespace\}$, we recall a few definitions from
probability theory that are essential to the presentation of our
results. We measure the convergence of the stochastic process $\statprob$
using one of two common statistical distances~\cite{Csiszar67}: the Hellinger
distance and the total variation distance (our definitions are a factor of 2
different from some definitions of these metrics). The total variation
distance between probability distributions $P$ and $Q$ defined on a set
$\statsamplespace$, assumed to have densities $p$ and $q$ with respect to an
underlying measure $\mu$,\footnote{This is no loss of generality, since $P$
  and $Q$ are absolutely continuous with respect to $P + Q$.} is
\begin{equation}
  \label{eqn:def-total-variation}
  \dtv(P, Q) \defeq \int_\statsamplespace |p(\statsample) -
  q(\statsample)| d\mu(\statsample) = 2 \sup_{A
    \subset \statsamplespace} |P(A) - Q(A)|,
\end{equation}
the supremum taken over measurable subsets of $\statsamplespace$.
The squared Hellinger distance is
\begin{equation}
  \label{eqn:def-hellinger}
  \dhel(P, Q)^2 = \int_\statsamplespace
  \left(\sqrt{\frac{p(\statsample)}{q(\statsample)}} - 1\right)^2
  q(\statsample) d\mu(\statsample) = \int_\statsamplespace
  \left(\sqrt{p(\statsample)} - \sqrt{q(\statsample)}\right)^2
  d\mu(\statsample).
\end{equation}
It is a well-known fact~\cite{Csiszar67} that for any probability
distributions $P$ and $Q$,
\begin{equation}
  \label{eqn:hellinger-tv-ordering}
  \dhel(P, Q)^2 \le \dtv(P, Q) \le 2 \dhel(P, Q).
\end{equation}

Using the total variation~\eqref{eqn:def-total-variation} and
Hellinger~\eqref{eqn:def-hellinger} metrics, we now describe our notion of
mixing (convergence) of the stochastic process $\statprob$.  Recall our
definition of the $\sigma$-field $\mc{F}_t = \sigma(\statsample_1, \ldots,
\statsample_t)$. Let $\statprob^t_{[s]}$ denote the distribution of
$\statsample_t$ conditioned on $\mc{F}_s$ (i.e.\ given the initial samples
$\statsample_1, \ldots, \statsample_s$), so for measurable
$A \subset \statsamplespace$ we have
$\statprob^t_{[s]}(A) \defeq \statprob(\statsample_t \in A \mid \mc{F}_s)$. We
measure convergence of $\statprob$ to $\stationary$ in terms of the mixing
time of the different $\statprob^t_{[s]}$, defined for the Hellinger and total
variation distances as follows. In the definitions, let $\statdensity^t_{[s]}$
and $\stationarydensity$ denote the densities of $\statprob^t_{[s]}$ and
$\stationary$, respectively.
\begin{definition}
  \label{def:tv-mixing-time}
  The \emph{total variation mixing time} $\tmixtv(\statprob_{[s]}, \epsilon)$
  of the sampling distribution $\statprob$ conditioned on the $\sigma$-field
  of the initial $s$ samples $\mc{F}_s = \sigma(\statsample_1, \ldots,
  \statsample_s)$ is the smallest $t \in \N$ such that
    $\dtv(\statprob_{[s]}^{s + t}, \stationary) \le \epsilon$,
  \begin{equation*}
    \tmixtv(\statprob_{[s]}, \epsilon) \defeq
    \inf\left\{t - s : t \in \N, ~ \int_{\statsamplespace}
    \left|\statdensity_{[s]}^t(\statsample) - \stationarydensity(\statsample)
    \right|
    d\mu(\statsample) \le \epsilon \right\}.
  \end{equation*}
  The \emph{Hellinger mixing time} $\tmixhel(\statprob_{[s]}, \epsilon)$
  is the smallest $t$ such that
  $\dhel(\statprob_{[s]}^{s + t}, \stationary) \le \epsilon$,
  \begin{equation*}
    \tmixhel(\statprob_{[s]}, \epsilon) \defeq
    \inf\left\{t - s : t \in \N, ~ \int_{\statsamplespace}
    \left(\sqrt{\statdensity_{[s]}^t(\statsample)} - 
    \sqrt{\stationarydensity(\statsample)}
    \right)^2
    d\mu(\statsample) \le \epsilon^2 \right\}.
  \end{equation*}
\end{definition}

Put another way, the mixing times $\tmixtv(\statprob_{[s]}, \epsilon)$ and
$\tmixhel(\statprob_{[s]}, \epsilon)$ are the number of \emph{additional}
steps required until the distribution of $\statsample_t$ is close to the
stationary distribution $\stationary$ given the initial $s$ samples
$\statsample_1, \ldots, \statsample_s$.

The following assumption, which makes the mixing times of the stochastic
process $\statprob$ uniform, is our main probabilistic assumption.
\begin{assumption}
  \label{assumption:uniform-mixing}
  The mixing times of the stochastic process $\{\statsample_i\}$ are uniform in
  the sense that there exist uniform mixing times $\tmixtv(\statprob,
  \epsilon), \tmixhel(\statprob, \epsilon) < \infty$ such that with
  probability $1$,
  \begin{equation*}
    \tmixtv(\statprob, \epsilon) \ge \tmixtv(\statprob_{[s]}, \epsilon)
    ~~~ \mbox{and} ~~~
    \tmixhel(\statprob, \epsilon) \ge \tmixhel(\statprob_{[s]}, \epsilon)
  \end{equation*}
  for all $\epsilon > 0$ and $s \in \N$.
\end{assumption}

Assumption~\ref{assumption:uniform-mixing} is a weaker version of the common
assumption of $\phi$-mixing in the probability
literature~(e.g.~\cite{Bradley05}); $\phi$-mixing
requires convergence of the process over the entire ``future'' $\sigma$-field
$\sigma(\statsample_t, \statsample_{t + 1}, \ldots)$ of the process
$\statsample_t$. Any finite state-space time-homeogeneous
Markov chain satisfies
the above assumption, as do uniformly ergodic Markov chains on general state
spaces~\cite{MeynTw09}.

We remark that the definition~\ref{def:tv-mixing-time} of mixing time does not
assume that the distributions $\statprob_{[s]}$ are time-homogeneous.  Indeed,
Assumption~\ref{assumption:uniform-mixing} requires only that there exists a
uniform upper bound on the mixing times. We can weaken
Assumption~\ref{assumption:uniform-mixing} to allow randomness in the
probability distributions $\statprob^t_{[s]}$ themselves, that is, conditional
on $\mc{F}_s$, the mxing time $\tmixtv(\statprob_{[s]}, \epsilon)$ is an
$\mc{F}_s$-measurable random variable. Our weakened probabilistic assumption
is
\begin{assumption}
  \label{assumption:probabilistic-mixing}
  The mixing times of the stochastic process $\{\statsample_i\}$ are
  stochastically uniform in the sense that there exists a uniform
  mixing time $\tmixtv(\statprob, \epsilon) < \infty$,
  continuous from the right as a function of $\epsilon$, such that
  for all $\epsilon > 0$, $s \in \N$, and $c \in \R$
  \begin{equation*}
    \statprob\!
    \left(\tmixtv(\statprob_{[s]}, \epsilon) \ge \tmixtv(\statprob, \epsilon)
    + \mixconst c\right) \le \exp(-c).
  \end{equation*}
\end{assumption}

Assumption~\ref{assumption:probabilistic-mixing}
allows us to provide convergence guarantees for a much wider range of
processes, such as auto-regressive processes, than permitted by
Assumption~\ref{assumption:uniform-mixing}.

%% file: main-results.tex
\section{Main results}
\label{sec:main-results}

With our assumptions in place, we can now give our main results.  We
begin with three general theorems that guarantee the convergence of
the EMD algorithm in expectation and with high probability. The second
part of the section shows that our analysis is sharp---unimprovable by
more than numerical constant factors---by giving an
information-theoretic lower bound on the convergence rate of any
optimization procedure receiving non-i.i.d.\ samples from $\statprob$.

\subsection{Convergence guarantees}

Our first result gives convergence in expectation
of the EMD algorithm~\eqref{eqn:ergodic-md}; we provide the proof
in \S~\ref{sec:proof-of-theorem-expected}.
\begin{theorem}
  \label{theorem:expected-convergence}
  Let Assumption~\ref{assumption:uniform-mixing} hold and let $x(t)$
  be defined by the EMD update~\eqref{eqn:ergodic-md} with
  non-increasing stepsize sequence $\{\stepsize(t)\}$. Let
  $x^\star \in \xdomain$ be arbitrary and let~\eqref{eqn:compactness}
  hold. If Assumption~\ref{assumption:one-step-variance} holds, then
  for any $\epsilon > 0$,
  \begin{align*}
    \lefteqn{\E\bigg[\sum_{t = 1}^T \left(f(x(t)) - f(x^\star)\right)\bigg]} \\
    & \quad \le \frac{\radius^2}{2 \stepsize(T)}
    + \frac{\lipobj^2}{2} \sum_{t=1}^T\stepsize(t)
    + 3 T \epsilon \lipobj \radius
    + (\tmixhel(\statprob, \epsilon)  - 1)
    \bigg[
      \lipobj^2\sum_{t=1}^T\stepsize(t) +
      \radius \lipobj\bigg],
  \end{align*}
  while if Assumption~\ref{assumption:F-lipschitz} holds, then for any
  $\epsilon > 0$,
  \begin{align*}
    \lefteqn{\E\bigg[\sum_{t=1}^T \left(f(x(t)) - f(x^\star)\right)\bigg]} \\
    & \quad \le \frac{\radius^2}{2 \stepsize(T)}
    + \frac{\lipobj^2}{2} \sum_{t=1}^T\stepsize(t)
    + T \epsilon \lipobj \radius
    + (\tmixtv(\statprob, \epsilon) - 1) \bigg[
      \lipobj^2\sum_{t=1}^T\stepsize(t) + \radius \lipobj\bigg].
  \end{align*}
  The expectation in both bounds is taken with respect to the
  samples $\statsample_1, \ldots, \statsample_T$.
\end{theorem}

We obtain an immediate corollary to Theorem~\ref{theorem:expected-convergence}
by applying Jensen's inequality to the convex function $f$:
\begin{corollary}
  \label{corollary:expected-convergence}
  Define $\what{x}(T) = \frac{1}{T} \sum_{t=1}^T x(t)$ and let the conditions
  of Theorem~\ref{theorem:expected-convergence} hold.  If
  Assumption~\ref{assumption:one-step-variance} holds, then for any $\epsilon
  > 0$
  \begin{equation*}
    \E[f(\what{x}(T)) - f(x^\star)]
    \le \frac{\radius^2}{2 \stepsize(T)T}
    + \frac{\lipobj^2}{2T} \sum_{t=1}^T\stepsize(t)
    + 3 \epsilon \lipobj \radius
    + \frac{\tmixhel(\statprob, \epsilon) - 1}{T} \bigg[
      \lipobj^2 \sum_{t=1}^T\stepsize(t) +
      \radius \lipobj\bigg].
  \end{equation*}
  If Assumption~\ref{assumption:F-lipschitz} holds, then for any $\epsilon >
  0$
  \begin{equation*}
    \E[f(\what{x}(T)) - f(x^\star)]
    \le \frac{\radius^2}{2 \stepsize(T) T}
    + \frac{\lipobj^2}{2T} \sum_{t = 1}^T \stepsize(t)
    + \epsilon \lipobj\radius + \frac{\tmixtv(\statprob, \epsilon) - 1}{T}
    \bigg[\lipobj^2 \sum_{t = 1}^T \stepsize(t) + \radius \lipobj \bigg].
  \end{equation*}
\end{corollary}

Corollary~\ref{corollary:expected-convergence} shows that so long as the
stepsize sequence $\stepsize(t)$ is non-increasing and satisfies the
asymptotic conditions $T \stepsize(T) \rightarrow \infty$ and $(1/T)
\sum_{t=1}^T \stepsize(t) \rightarrow 0$, the EMD method converges. We can
also provide similar high-probability convergence guarantees:
\begin{theorem}
  \label{theorem:highprob-convergence}
  Let the conditions of Theorem~\ref{theorem:expected-convergence} and
  Assumption~\ref{assumption:F-lipschitz} hold. Let $\delta \in (0, 1)$
  and define the average $\what{x}(T) = \frac{1}{T} \sum_{t = 1}^T x(t)$.
  With probability at least $1 - \delta$, for $\epsilon > 0$
  such that $\tmixtv(\statprob, \epsilon) \le T/2$,
  \begin{align*}
    f(\what{x}(T)) - f(x^\star)
    & \le \frac{\radius^2}{2 T \stepsize(T)} +
    \frac{\lipobj^2}{2T} \sum_{t = 1}^T \stepsize(t)
    + \frac{\tmixtv(\statprob, \epsilon) - 1}{T}\bigg[
      \lipobj^2 \sum_{t = 1}^{T} \stepsize(t) +
      \lipobj \radius \bigg] \\
    & \qquad\quad ~
    + \epsilon \lipobj \radius
    + 4 \lipobj \radius \sqrt{\frac{\tmixtv(\statprob, \epsilon)
        \log \frac{\tmixtv(\statprob, \epsilon)}{\delta}}{T}}.
  \end{align*}
\end{theorem}

We provide the proof of this theorem in
\S~\ref{sec:proof-of-theorem-highprob}. Note that the rate of convergence
in Theorem~\ref{theorem:highprob-convergence} is identical to that obtained in
Theorem~\ref{theorem:expected-convergence} plus an additional term that arises
as a result of the control of the deviation of the ergodic process around its
expectation.  The additional $\log\frac{1}{\delta}$-dependent term arises from
the application of martingale concentration inequalities~\cite{Azuma67}, which
requires some care because the process $\{\statsample_t\}$ is coupled over
time. Nonetheless, as we discuss briefly following
Corollary~\ref{corollary:geometric-convergence}---and as made clear by our
lower bound in Theorem~\ref{theorem:lower-bound}---the additional terms
introduce a factor of at most $\sqrt{\log \tmixtv(\statprob, \epsilon)}$ to
the bounds. That is, the dominant terms in the convergence rates (modulo
logarithmic factors) also appear in the expected bounds in
Theorem~\ref{theorem:expected-convergence}.

The last of our convergence theorems extends the previous
two to the case when the stochastic process is not uniformly mixing, but
has mixing properties that may depend on its state. We provide
the proof of Theorem~\ref{theorem:probabilistic-mixing} in
\S~\ref{sec:proof-of-theorem-probabilistic-mixing}.
\begin{theorem}
  \label{theorem:probabilistic-mixing}
  Let the conditions of Theorem~\ref{theorem:highprob-convergence} hold,
  except that we replace the uniform mixing
  assumption~\ref{assumption:uniform-mixing} with the probabilistic mixing
  assumption~\ref{assumption:probabilistic-mixing}.  Let $\delta \in (0,
  1)$. In the notation of Assumption~\ref{assumption:probabilistic-mixing},
  define
  \begin{equation*}
    \tau(\epsilon, \delta)
    \defeq
    \tmixtv(\statprob, \epsilon) + \mixconst\left(\log\frac{2}{\delta} + 2
    \log(T)\right).
  \end{equation*}
  With probability at least $1 - \delta$, for any $x^\star \in \xdomain$,
  \begin{align*}
    f(\what{x}(T)) - f(x^\star)
    & \le \inf_{\epsilon > 0} \bigg\{ \frac{\radius^2}{2 T \stepsize(T)}
    + \frac{\lipobj^2}{2T} \sum_{t = 1}^T \stepsize(t)
    + \frac{\tau(\epsilon, \delta) - 1}{T}
    \bigg[\lipobj^2 \sum_{t = 1}^T \stepsize(t) + \lipobj \radius\bigg] \\
    & \qquad\qquad ~
    + \epsilon \lipobj \radius 
    + 4 \lipobj\radius \sqrt{\frac{\tau(\epsilon, \delta)
        \log\frac{\tau(\epsilon, \delta)}{\delta}}{T}}
    \,\bigg\}.
  \end{align*}
\end{theorem}

In \S~\ref{sec:probabilistic-mixing-examples} we give two applications of
Theorem~\ref{theorem:probabilistic-mixing} (to estimation in autoregressive
processes and a fault-tolerant distributed optimization scheme) that show how
it makes the applicability of our development substantially broader.

We now turn to a slight specialization of our bounds to build
intuition and attain a simplified statement of convergence rates.
Theorems~\ref{theorem:expected-convergence},
\ref{theorem:highprob-convergence},
and~\ref{theorem:probabilistic-mixing} hold for essentially any
ergodic process that converges to the stationary distribution
$\stationary$. For a large class of processes, the convergence of the
distributions $\statprob^t$ to the stationary distribution
$\stationary$ is uniform and at a geometric rate~\cite{MeynTw09}:
there exist constants $\mixconst_1$ and $\mixconst_2$ such that
$\tmixtv(\statprob, \epsilon) \le
\mixconst_1\log(\mixconst_2/\epsilon)$. We have the following
corollary for this special case; we only present the version yielding
expected convergence rates, as the high-probability corollary is
similar.  In addition, by the fact~\eqref{eqn:hellinger-tv-ordering}
relating $\dhel$ to $\dtv$, if the process $\statprob$ satisfies
$\tmixtv(\statprob, \epsilon) \le \mixconst_1 \log(\mixconst_2 /
\epsilon)$, then there exist constants $\mixconst_1'$ and
$\mixconst_2'$ such that $\tmixhel(\statprob, \epsilon) \le
\mixconst_1' \log(\mixconst_2' / \epsilon)$. Thus we only state the
corollary for total variation mixing and under
Assumption~\ref{assumption:F-lipschitz}; an analogous result holds
under Assumption~\ref{assumption:one-step-variance} for mixing with
respect to the Hellinger distance.
\begin{corollary}
  \label{corollary:geometric-convergence}
  Under the conditions of Theorem~\ref{theorem:expected-convergence},
  assume in addition that \mbox{$\tmixtv(\statprob, \epsilon) \le
    \mixconst_1 \log(\mixconst_2/\epsilon)$} and let
  Assumption~\ref{assumption:F-lipschitz} hold. The EMD
  update~\eqref{eqn:ergodic-md} with stepsize $\stepsize(t) =
  \stepsize/\sqrt{t}$ satisfies
  \begin{equation*}
    \E\left[f(\what{x}(T)) - f(x^\star) \right] \leq
    \frac{\radius^2}{2\stepsize\sqrt{T}} + \frac{2 \stepsize
      \lipobj^2}{\sqrt{T}} \Big(\mixconst_1
    \log\frac{\mixconst_2}{\epsilon}\Big) + \epsilon \lipobj\radius
    + \frac{\radius \lipobj
      \mixconst_1\log\frac{\mixconst_2}{\epsilon}}{T}.
  \end{equation*}
\end{corollary}
\begin{proof}
  Using the definition $\stepsize(t) = \stepsize / \sqrt{t}$ and the
  integral bound
  \begin{equation}
    \sum_{t=1}^T \frac{1}{\sqrt{t}}
    \le 1 + \int_1^T t^{-1/2} dt = 
    2 \sqrt{T} - 1 < 2 \sqrt{T},
    \label{eqn:sqrt-integral-bound}
  \end{equation}
  we have $\sum_{t=1}^T \stepsize(t) \le 2 \stepsize \sqrt{T}$. The
  corollary now follows from
  Theorem~\ref{theorem:expected-convergence}.
\end{proof}

We can obtain a simplified convergence rate with appropriate choice of the
stepsize multiplier $\stepsize$ and mixing parameter $\epsilon$: choosing
$\stepsize = \radius / (\lipobj \sqrt{\mixconst_1 \log(\mixconst_2 T)})$ and
$\epsilon = T^{-1/2}$ reduces the corollary to
\begin{equation}
  \E[f(\what{x}(T)) - f(x^\star)] =
  \order\left(\frac{\radius \lipobj \sqrt{\mixconst_1 \log(\mixconst_2 T)}}{
    \sqrt{T}}\right).
  \label{eqn:simple-rate-geometric}
\end{equation}

More generally, using the stepsize $\stepsize(t) = \stepsize /
\sqrt{t}$ and the same argument as in
Corollary~\ref{corollary:geometric-convergence} gives
\begin{equation}
  \E[f(\what{x}(T)) - f(x^\star)] \le \inf_{\epsilon > 0} \left\{
  \frac{\radius^2}{2 \stepsize \sqrt{T}} + \frac{2 \stepsize
    \lipobj^2}{\sqrt{T}} \tmixtv(\statprob, \epsilon) + \epsilon
  \lipobj \radius + \frac{\radius \lipobj (\tmixtv(\statprob,
    \epsilon) - 1)}{T} \right\}.
  \label{eqn:simple-rate}
\end{equation}
Again choosing $\epsilon = T^{-1/2}$ and defining the shorthand $\tmix =
\tmixtv(\statprob, T^{-1/2})$, by choosing $\stepsize = \radius / (\lipobj
\sqrt{\tmix})$, we see the bound~\eqref{eqn:simple-rate} implies that
\begin{equation}
  \E[f(\what{x}(T)) - f(x^\star)]
  \le \frac{5 \radius \lipobj}{2} \cdot \frac{\sqrt{\tmix}}{\sqrt{T}}
  + \frac{\radius \lipobj}{\sqrt{T}}
  + \frac{\radius \lipobj (\tmix - 1)}{T}.
  \label{eqn:super-simple-rate}
\end{equation}
In the classical setting~\cite{NemirovskiJuLaSh09} of i.i.d.\ samples
$\statsample \sim \stationary$, stochastic gradient descent and its
mirror descent generalizations attain convergence rates of
$\order(R\lipobj/\sqrt{T})$. Since $\tmixtv(\statprob, 0) =
\tmixhel(\statprob, 0) = 1$ for an i.i.d.\ process, the
rate~\eqref{eqn:simple-rate} shows that our results subsume existing
results for i.i.d.\ noise. Moreover, they are sharp in the
i.i.d.\ case, that is, unimprovable by more than a numerical constant
factor~\cite{NemirovskiYu83,AgarwalBaRaWa12}.

In addition, we note that the conclusions of
Corollary~\ref{corollary:geometric-convergence} (and the
bound~\eqref{eqn:simple-rate}) hold---modulo an additional $\log
\tmixtv(\statprob, \epsilon)$---with high probability. We may also
note that replacing $\epsilon \lipobj \radius$ with $3 \epsilon
\lipobj \radius$ and $\tmixtv$ with $\tmixhel$ in the
bound~\eqref{eqn:simple-rate} yields a guarantee under
Assumption~\ref{assumption:one-step-variance}. Further, the step-size
choice $\stepsize(t) = \stepsize / \sqrt{t}$ is robust---in a way
similarly noted by Nemirovski et al.~\cite{NemirovskiJuLaSh09}---for
quickly mixing ergodic processes. Indeed, using the
inequalities~\eqref{eqn:simple-rate}
and~\eqref{eqn:super-simple-rate}, we see that setting the multiplier
$\stepsize = \gamma \radius / (\lipobj \sqrt{\tmix})$ yields
$\E[f(\what{x}(T)) - f(x^*)] = \order(\max\{\gamma, \gamma^{-1}\}
\radius \lipobj \sqrt{\tmix} / \sqrt{T})$, so mis-specification of
$\stepsize$ by a constant $\gamma$ leads to a penalty in convergence
that scales at worst linearly in $\max\{\gamma^{-1}, \gamma\}$.  In
classical stochastic approximation
settings~\cite{RobbinsMo51,Spall03,KushnerYi03}, one usually chooses
step size sequence $\stepsize(t) = \order(t^{-m})$ for $m \in
\openleft{.5}{1}$; in our case, such choices may yield sub-optimal
rates because we study convergence of the averaged parameter
$\what{x}(T)$ rather than the final parameter $x(t)$. Nonetheless, averaging
is known to yield robustness in
i.i.d.\ settings~\cite{PolyakJu92,NemirovskiJuLaSh09}, and moreover
gives unimprovable convergence rates in many cases (see
\S~\ref{sec:lower-bounds} as well as
references~\cite{NemirovskiYu83,AgarwalBaRaWa12}).  We provide some
evidence of this robustness in numerical simulations in
\S~\ref{sec:experiments}, and we see generally that EMD has
qualitative convergence behavior similar to stochastic mirror descent
for a broad class of ergodic processes.

Before continuing, we make two final remarks. First, none of our main
theorems assume Markovianity or even homogeneity of the stochastic
process $\statprob$; all that is needed is that the mixing time
$\tmixtv$ (or $\tmixhel$) exists, or even that it exists only with
some reasonably high probability. Previous work similar to
ours~\cite{RamNeVe09a,JohanssonRaJo09} assumes Markovianity (see also
our discussion concluding \S~\ref{sec:probabilistic-mixing-examples}).
Further, general ergodic processes do not always enjoy the geometric
mixing assumed in Corollary~\ref{corollary:geometric-convergence},
satisfying either Assumption~\ref{assumption:probabilistic-mixing}'s
probabilistic mixing condition or simply mixing more slowly.  In
\S~\ref{sec:probabilistic-mixing-examples}, we present examples of
such probabilistically mixing processes on general state spaces, while
the bound~\eqref{eqn:simple-rate} suggests an approach to attain
convergence for more slowly mixing processes (see
\S~\ref{sec:slow-mixing}).

\subsection{Lower bounds and optimality guarantees}
\label{sec:lower-bounds}

Our final main result concerns the optimality of the results we have
presented. Informally, the theorem states that our results are
unimprovable by more than numerical constant factors, though making
this formal requires additional notation. In the stochastic gradient
oracle model of convex
optimization~\cite{NemirovskiYu83,AgarwalBaRaWa12}, a method $\method$
issues queries of the form $x \in \xdomain$ to an oracle that returns
noisy function and gradient information. In our setting, the oracle is
represented by the pair $\oracle = (\statprob, \gradfunc)$, and when
the oracle is queried at a point $x$ at time $t$ (i.e., this is the
$t$th query $\oracle$ has received), it draws a sample $\statsample_t$
according to the distribution $\statprob(\cdot \mid \statsample_1,
\ldots, \statsample_{t-1})$ and returns $\gradfunc(x, \statsample_t)
\in \R^d$. The method issues a sequence of queries $x(1), \ldots,
x(t)$ to the oracle and may use $\{\gradfunc(x(1), \statsample_1),
\ldots, \gradfunc(x(t), \statsample_t)\}$ to devise a new query point
$x(t + 1)$. For an oracle $\oracle$, we define the error of the method
$\method$ on a function $f$ after $T$ queries of the oracle as
\begin{equation}
  \optgap_T(\method, f, \xdomain, \oracle)
  = f(\what{x}) - \inf_{x \in \xdomain} f(x),
  \label{eqn:optgap-def}
\end{equation}
where $\what{x}$ denotes the method $\method$'s estimate of the
minimizer of $f$ after seeing the $T$ samples $\{\gradfunc(x(1),
\statsample_1), \ldots, \gradfunc(x(T), \statsample_T)\}$. The
quantity~\eqref{eqn:optgap-def} is random, so we measure accuracy in
terms of the expected value $\E_\oracle[\optgap_T(\method, f,
  \xdomain, \oracle)]$, where the expectation is taken with respect to
the randomness in $\oracle$.

Now we define a natural collection of stochastic oracles for
our dependent setting.
\begin{definition}
  \label{def:oracle-set}
  For $f$ convex, $\tau \in \N$, $\lipobj \in (0, \infty)$, and
  $p \in [1, \infty]$,
  the \emph{admissible oracle set} $\oracleset{f, \tau, \lipobj, p}$ is the
  set of oracles $\oracle = (\statprob, \gradfunc)$ for which there exists a
  probability distribution $\stationary$ on $\statsample$ such that
  \begin{align*}
    & \norm{\gradfunc(x; \statsample)}_p \le \lipobj
    ~ \mbox{for~} x \in \xdomain ~ \mbox{and} ~
    \statsample \in \statsamplespace,
    ~~~
    \E_\stationary[\gradfunc(x; \statsample)] \in \partial f(x)
    ~ \mbox{for~} x \in \xdomain, \\
    & ~~~\mbox{and} ~~
    \dtv\left(\statprob_{[t]}^{t + \tau}, \stationary
    \right) = 0 ~ \mbox{for~all~} t \in \N
    ~ \mbox{with~probability~} 1.
  \end{align*}
\end{definition}

The set $\oracleset{f,\tau,\lipobj,p}$ is the collection of oracles
$\oracle = (\statprob, \gradfunc)$ for which the distribution
$\statprob$ has stationary distribution $\stationary$, mixing time
bounded by $\tau$, and returns $\ell_p$-norm bounded stochastic
subgradients of the function $f$. The condition $\norm{\gradfunc(x;
  \statsample)}_p \le \lipobj$ guarantees that
Assumptions~\ref{assumption:one-step-variance}
and~\ref{assumption:F-lipschitz} hold, while $\dtv(\statprob_{[t]}^{t
  + \tau}, \stationary) = 0$ satisfies
Assumption~\ref{assumption:uniform-mixing}.  With
Definition~\ref{def:oracle-set}, for any collection $\functions$ of
convex functions $f$, we can define the minimax error over
distributions with mixing times bounded by $\tau$ as
\begin{equation}
  \label{eqn:minimax-def}
  \optgap_T^*(\functions, \xdomain, \tau, \lipobj, p)
  \defeq \inf_{\method} ~ \sup_{f \in \functions}
  \sup_{\oracle \in \oracleset{f, \tau, \lipobj, p}}
  \E_\oracle\left[\optgap_T(\method, f, \xdomain, \oracle)\right].
\end{equation}
We have the following theorem on this minimax error (see
\S~\ref{sec:lower-bound-proof} for a proof).
\begin{theorem}
  \label{theorem:lower-bound}
  Let $\xdomain \subseteq \R^d$ be a convex set containing the
  $\ell_\infty$ ball of radius $\linfradius$ for some $\linfradius >
  0$.  Let $1 /p + 1/q = 1$ and $p \ge 1$ and let the set $\functions$
  consist of convex functions that are $\lipobj$-Lipschitz continuous
  with respect to the $\ell_q$-norm over the set $\xdomain$.  For $p
  \in [1, 2]$ and for any $\tau \in \N$, the minimax oracle
  complexity~\eqref{eqn:minimax-def} satisfies
  \begin{subequations}
    \begin{equation}
      \optgap_T^*(\functions, \xdomain, \tau, \lipobj, p)
      = \Omega\left(\lipobj \linfradius \sqrt{d} \sqrt{\frac{\tau}{T}}\right).
      \label{eqn:sgd-lower-bound}
    \end{equation}
    For $p \in [2, \infty]$ and for any $\tau \in \N$, the minimax oracle
    complexity~\eqref{eqn:minimax-def} satisfies
    \begin{equation}
      \optgap_T^*(\functions, \xdomain, \tau, \lipobj, p)
      = \Omega\left(\lipobj \linfradius d^{\frac{1}{q}} \sqrt{\frac{\tau}{
        T}}\right).
      \label{eqn:emd-lower-bound}
    \end{equation}
  \end{subequations}
\end{theorem}

We make a few brief comments on the implications of
Theorem~\ref{theorem:lower-bound}.  First, the dependence on $\tau$
and $T$ in the bounds of $\sqrt{\tau / T}$ matches that of the upper
bound~\eqref{eqn:super-simple-rate}. In addition, following the
discussion of Agarwal et al.~\cite[Section~III.A and
  Appendix~C]{AgarwalBaRaWa12}, we can see that the dependence of the
bounds~\eqref{eqn:sgd-lower-bound} and~\eqref{eqn:emd-lower-bound} on
the quantities $\linfradius$, $\lipobj$, and the dimension $d$ are
optimal (to within logarithmic factors). In brief, the
bound~\eqref{eqn:sgd-lower-bound} is achieved by taking $\prox(x) =
\half \ltwo{x}^2$ in the definition of the proximal function for the
EMD algorithm, while the bound~\eqref{eqn:emd-lower-bound} is achieved
by taking $\prox(x) = \half \norm{x}_q^2$ for $q = 1 + 1 / \log(d)$
(see also~\cite[Section 5]{Ben-TalMaNe01,BeckTe03}). Summarizing, we
find that
Theorems~\ref{theorem:expected-convergence}--\ref{theorem:probabilistic-mixing}
are unimprovable by more than numerical constants, and the EMD
algorithm~\eqref{eqn:ergodic-md} attains the minimax optimal rate of
convergence.

%% file: examples.tex
\section{Examples and Consequences}
\label{sec:examples}

We now collect several consequences of the convergence rates of
Theorems~\ref{theorem:expected-convergence},
\ref{theorem:highprob-convergence}, and~\ref{theorem:probabilistic-mixing} to
provide insight and illustrate applications of the theoretical statements.  We
begin with a concrete example and move toward more abstract principles,
completing the section with finite sample and asymptotic convergence
guarantees for more slowly mixing ergodic processes.
\ifdefined\siam
Most of the results are new or improve over previously known
bounds, and we provide a few additional examples in the extended version of
this paper~\cite{DuchiAgJoJo11}.
\else
Most of the results are new or improve over previously known
bounds.
\fi

\subsection{Peer-to-peer optimization and Markov incremental
gradient descent}
\label{sec:migd}

The Markov incremental gradient descent (MIGD) procedure due to
Johansson et al.~\cite{JohanssonRaJo09} is a generalization of
Nedi\'{c} and Bertsekas's randomized incremental subgradient
method~\cite{NedicBe01}, which Ram et al.~\cite{RamNeVe09a} further
analyze. The motivation for MIGD comes from a distributed
optimization algorithm using a simple (locally computable)
peer-to-peer communication scheme. We assume we have $n$ processors or
computers, each with a convex function $f_i : \xdomain \rightarrow
\R$, and the goal is to minimize
\begin{equation}
  \label{eqn:markov-problem}
  f(x) = \ninv \sum_{i=1}^n f_i(x)
  ~~~ \subjectto ~~~ x \in \xdomain.
\end{equation}
The procedure works as follows.  The current set of parameters $x(t) \in
\xdomain$ is passed among the processors in the network, where a token
$\token(t) \in [n]$ indicates the processor holding $x(t)$ at iteration $t$.
At iteration $t$, the algorithm computes the update
\begin{equation*}
  g(t) \in \partial f_{\token(t)}(x(t)),
  ~~~~
  x(t + 1) = \argmin_{x \in \xdomain} \left\{\<g(t), x\>
  + \frac{1}{\stepsize(t)} \divergence(x, x(t))\right\},
\end{equation*}
after which the token $\token(t)$ moves to a new processor.  This update is a
generalization of the papers~\cite{JohanssonRaJo09,RamNeVe09a}, which assume
$\prox(x) = \half \ltwo{x}^2$. Slightly more generally, the local functions
may be defined as expectations, $f_i(x) = \E_{\stationary_i}[F(x;
  \statsample)]$, for a local distribution $\stationary_i$. At iteration $t$, a
sample $\statsample_{t,\token(t)}$ is drawn from the local distribution
$\stationary_{\token(t)}$ and the algorithm computes the update
\begin{equation}
  g(t) \in \partial F(x(t); \statsample_{t,\token(t)}),
  ~~~
  x(t + 1) = \argmin_{x \in \xdomain} \left\{\<g(t), x\>
  + \frac{1}{\stepsize(t)} \divergence(x, x(t))\right\}.
  \label{eqn:general-markov-algorithm}
\end{equation}

We view the token $\token(t)$ as evolving according to a Markov
chain with doubly-stochastic transition matrix $\stochmat$, so its
stationary distribution is the uniform distribution.  In this case,
\begin{equation*}
  \statprob(\token(t) = j \mid \token(t - 1) = i) = \stochmat_{ij}.
\end{equation*}
The total variation distance of the stochastic process
initialized at $\token(0) = i$ from the true (uniform) distribution is
$\lone{P^t e_i - \onevec / n}$, where $e_i$ denotes the $i$th standard basis
vector. In addition, since $\stochmat$ is doubly stochastic, we have
$\stochmat \onevec = \onevec$ and thus
\ifdefined\siam
\begin{align*}
  \lone{\stochmat^t e_i - \onevec/n}
  \le \sqrt{n} \ltwo{\stochmat^t e_i - \onevec/n}
  & = \sqrt{n} \ltwo{P^t(e_i - \onevec)} \\
  & \le \sqrt{n} \singval_2(\stochmat)^t \ltwo{e_i - \onevec/n}
  \le \sqrt{n} \singval_2(\stochmat)^t,
\end{align*}
\else
\begin{equation*}
  \lone{\stochmat^t e_i - \onevec/n}
  \le \sqrt{n} \ltwo{\stochmat^t e_i - \onevec/n}
  = \sqrt{n} \ltwo{P^t(e_i - \onevec)}
  \le \sqrt{n} \singval_2(\stochmat)^t \ltwo{e_i - \onevec/n}
  \le \sqrt{n} \singval_2(\stochmat)^t,
\end{equation*}
\fi
where $\singval_2(\stochmat)$ denotes the second singular value of the matrix
$\stochmat$. From this spectral bound on the total variation distance, we see
that if $t \ge \frac{\half \log(Tn)}{\log \singval_2(\stochmat)^{-1}}$ we
have $\lone{\stochmat^t e_i - \onevec/n} \le \frac{1}{\sqrt{T}}$. In addition,
recalling the sandwich
inequalities~\eqref{eqn:hellinger-tv-ordering}, we have
\begin{equation*}
  \dhel(\stochmat^t e_i, \onevec/n)
  \le
  \sqrt{\dtv(\stochmat^t e_i, \onevec / n)}
  \le
  n^{1/4} \singval_2(\stochmat)^{t/2}
\end{equation*}
so $\dhel(\stochmat^t e_i, \onevec/n) \le 1/\sqrt{T}$ when $t \ge
\frac{\log(Tn)}{\log \singval_2(\stochmat)^{-1}}$.  In the notation of
Assumption~\ref{assumption:uniform-mixing},
\begin{equation}
  \tmixtv(\stochmat, T^{-1/2})
  \le \frac{\log (Tn)}{2 \log \singval_2(\stochmat)^{-1}}
  \le \frac{\log(Tn)}{2(1 - \singval_2(\stochmat))}
  ~~~ \mbox{and} ~~~
  \tmixhel(\stochmat, T^{-1/2}) \le \frac{\log(Tn)}{1 - \singval_2(\stochmat)}.
  \label{eqn:markov-chain-tmix}
\end{equation}
(Since $\log \singval^{-1} \approx 1 - \singval$ for $\singval \approx 1$,
using $1 - \singval$ is no significant loss in our applications.)
Consequently, we have the following result, similar to
Corollary~\ref{corollary:geometric-convergence}.
\begin{corollary}
  \label{corollary:migd}
  Let $x(t)$ evolve according to the Markov incremental descent
  update~\eqref{eqn:general-markov-algorithm}, where $\token(t)$ evolves via
  the doubly stochastic transition matrix $\stochmat$ and
  $\stepsize(t)=\stepsize/\sqrt{t}$.  Define $\what{x}(T) = \frac{1}{T}
  \sum_{t=1}^T x(t)$ and $\tmix = \sqrt{\log(Tn)} / \sqrt{1 -
    \singval_2(\stochmat)}$.  Choose stepsize multiplier $\stepsize = \radius
  / \lipobj \sqrt{\tmix}$.  If for each distribution $\stationary_i$ we have
  $\E_{\stationary_i}[\dnorm{\gradfunc(x; \statsample)}^2] \le \lipobj^2$,
  then
  \begin{equation}
    \E[f(\what{x}(T))] - f(x^\star)
    \le \frac{5\radius \lipobj}{2} \cdot
    \frac{\sqrt{\tmix}}{\sqrt{T}}
    + \frac{5 \radius \lipobj}{\sqrt{T}}
    + \frac{\radius \lipobj}{T} \cdot \tmix.
    \label{eqn:migd-expectation}
  \end{equation}
  Let $\delta \in (0, 1)$ and assume $\tmix \le T/2$.
  If for each $i$ and $\stationary_i$-almost every $\statsample$
  we have $\dnorm{\gradfunc(x; \statsample)} \le \lipobj$,
  then with probability at least $1 - \delta$
  \begin{equation*}
    f(\what{x}(T)) - f(x^\star)
    \le \frac{5\radius \lipobj}{2} \cdot \frac{\sqrt{\tmix}}{\sqrt{T}}
    + \frac{2 \radius \lipobj}{\sqrt{T}}
    + \frac{\radius \lipobj}{T} \cdot \tmix
    + \frac{3 \radius \lipobj}{\sqrt{T}}
    \sqrt{\tmix \log\frac{\tmix}{\delta}}.
  \end{equation*}
\end{corollary}
\begin{proof}
  The proof is a consequence of
  Theorems~\ref{theorem:expected-convergence}
  and~\ref{theorem:highprob-convergence} and
  Corollary~\ref{corollary:geometric-convergence}. We use the uniform
  bound~\eqref{eqn:markov-chain-tmix} on the mixing time of the random walk,
  in Hellinger or total variation distance, and the result follows via
  algebra.
\end{proof}

Corollary~\ref{corollary:migd} gives convergence rates sharper and somewhat
more powerful than those in the original Markov incremental gradient descent
papers~\cite{JohanssonRaJo09,RamNeVe09a}. First, our results allow us to use
mirror descent updates, thus applying to problems having non-Euclidean
geometry; it is by now well known that this is essential for obtaining
efficient methods for high-dimensional
problems~\cite{NemirovskiYu83,Ben-TalMaNe01,BeckTe03}. Secondly, because we
base our convergence analysis on mixing time rather than return times, we can
give sharp high-probability convergence guarantees.
Finally, our convergence rates are often tighter. Ram et al.~\cite{RamNeVe09a}
do not appear to give finite sample convergence rates, and as discussed
by Duchi et al.~\cite{DuchiAgWa12}, Johansson et al.~\cite{JohanssonRaJo09}
show that MIGD---with optimal choice of their algorithm
parameters---has convergence rate $\order(\radius \lipobj \max_i
\sqrt{\frac{n \Gamma_{ii}}{T}})$, where $\Gamma$ is the return time matrix
given by $\Gamma = (I - \stochmat + \onevec \onevec^\top/n)^{-1}$. When
$\stochmat$ is symmetric (as in~\cite[Lemma 1]{JohanssonRaJo09}), the
eigenvalues of $\Gamma$ are $1$ and $1 / (1 - \lambda_i(\stochmat))$ for $i >
1$, and
\begin{equation*}
  n \max_{i \in [n]} \Gamma_{ii} \ge \tr(\Gamma)
  = 1 + \sum_{i=2}^n \frac{1}{1 - \lambda_i(\stochmat)}
  > \frac{1}{1 - \singval_2(\stochmat)}.
\end{equation*}
Thus, up to logarithmic factors, the bound~\eqref{eqn:migd-expectation} from
Corollary~\ref{corollary:migd} is never weaker. For well-connected graphs, the
bound is substantially stronger; for example, a random walk on an expander
graph has constant spectral gap~\cite{Chung98}, so $(1 -
\singval_2(\stochmat))^{-1} = \order(1)$, while the previous bound is $n
\max_{i \in [n]} \Gamma_{ii} = \Omega(n)$.

\ifdefined\removecombinatorial
\else
\subsection{Optimization over combinatorial spaces}

For our second example, we consider settings where $\statsamplespace$ is a
combinatorial space from which it is difficult to obtain uniform samples but
for which we can construct a Markov chain that converges quickly to the
uniform distribution over $\statsamplespace$. See Jerrum and 
Sinclair~\cite{JerrumSi96} for an overview of such problems.
More concretely, consider the statistical problem of
learning a ranking function for web searches. The statistician receives
information in the form of a user's clicks on particular search results, which
impose a partial order on the results (since only a few are clicked on).  We
would like the resulting ranking function to be oblivious to the order of the
remaining results, which leads us to define $\statsamplespace$ to be the set of
all total orders of the search results consistent with the partial order imposed
by the user. Certainly the set $\statsamplespace$ is exponentially large; it
is also challenging to draw a uniform sample from it.

Though sampling is challenging, it is possible to develop a rapidly-mixing
Markov chain whose stationary distribution is uniform on
$\statsamplespace$. Specifically, Karzanov and Khachiyan~\cite{KarzanovKh91}
develop the following Markov chain. Let $\mc{P}$ be a partial order on the set
$[n]$, whose elements are of the form $i \prec j$ for $i, j \in [n]$.  The
states of the Markov chain are permutations $\sigma$ of $[n]$ respecting the
partial order $\mc{P}$, and the Markov chain transitions between permutations
$\sigma$ and $\sigma'$ by randomly selecting a pair $i, j \in [n]$, then
swapping their orders if this is consistent with the partial order $\mc{P}$.
Wilson~\cite{Wilson04} showed that the mixing time of this Markov chain is
bounded by
\begin{equation}
  \label{eqn:linear-extension-mixtime}
  \tmixtv(\statprob, \epsilon) \le \frac{4}{\pi^2} n^3 \log\frac{n}{\epsilon}.
\end{equation}
Similar results hold for sampling from other combinatorial
spaces~\cite{JerrumSi96}.

Theorem~\ref{theorem:highprob-convergence} gives the following consequence of
the bound~\eqref{eqn:linear-extension-mixtime} on the mixing time of the
Karzanov-Khachiyan Markov chain.  Denote the set of permutations $\sigma$
consistent with the partial order $\mc{P}$ by $\sigma \in \mc{P}$, so the
objective~\eqref{eqn:f-def} has the form
\begin{equation*}
  f(x) \defeq \frac{1}{\card(\sigma \in \mc{P})} \sum_{\sigma \in \mc{P}}
  F(x; \sigma).
\end{equation*}
We have
\begin{corollary}
  Let $x(t)$ evolve according to the EMD update~\eqref{eqn:ergodic-md}, where
  the sample space is the set of permutations $\{\sigma\}$ consistent with the
  partial order $\mc{P}$ over $[n]$. Define $\what{x}(T) = \frac{1}{T}
  \sum_{t=1}^T x(t)$. Under Assumption~\ref{assumption:F-lipschitz} with
  $\stepsize(t) = \stepsize/\sqrt{t}$ and the choice of multiplier
  $\stepsize = \pi \radius / 2 \lipobj \sqrt{\log(T n)}$,
  \ifdefined\siam
  \begin{align*}
    f(\what{x}(T)) - f(x^\star)
    & \le \frac{5\lipobj \radius}{2} \cdot
    \frac{n^{3/2}\sqrt{\log(T n)}}{\sqrt{T}}
    + \frac{\radius \lipobj n^3 \log(Tn)}{2T} \\
    & \qquad ~ + \frac{3 \lipobj\radius}{\sqrt{T}}
    \cdot \sqrt{n^3 \log (Tn)( \log[(n/\delta) \log(Tn)])}
  \end{align*}
  \else
  \begin{equation*}
    f(\what{x}(T)) - f(x^\star)
    \le \frac{5\lipobj \radius}{2} \cdot
    \frac{n^{3/2}\sqrt{\log(T n)}}{\sqrt{T}}
    + \frac{\radius \lipobj n^3 \log(Tn)}{2T}
    + \frac{4 \lipobj\radius}{\sqrt{T}}
    \cdot \sqrt{n^3 \log (Tn)( \log[(n/\delta) \log(Tn)])}
  \end{equation*}
  \fi
  with probability at least $1 - \delta$, where $\delta \in (0, 1)$.
\end{corollary}
\fi %

\subsection{Probabilistically mixing processes}
\label{sec:probabilistic-mixing-examples}
\newcommand{\simplex}{\Delta}

We now turn to two examples to show the broader applicability of the EMD
algorithm guaranteed by Theorem~\ref{theorem:probabilistic-mixing}.  Our first
example generalizes the Markov incremental gradient method of
\S~\ref{sec:migd} to allow random communication matrices $\statprob$, while
our second considers optimization problems where the data comes from a
(potentially nonlinear) autoregressive moving average (ARMA) process. For both
examples, we require a conversion from expected convergence of the total
variation distance $\dtv(\statprob_{[t]}^{t + \tau}, \stationary)$ as $\tau
\rightarrow \infty$ to the probabilistic bound in
Assumption~\ref{assumption:probabilistic-mixing}. To that end, we prove the
following lemma in Appendix~\ref{appendix:probabilistic-mixing}.
\begin{lemma}
  \label{lemma:beta-mixing-to-probabilities}
  Let $\E[\dtv(\statprob_{[t]}^{t + \tau}, \stationary)] \le K \singval^\tau$
  for all $\tau \in \N$, where $K \ge 1$ and $\singval \in (0, 1)$.
  Define
  \begin{equation*}
    \tmixtv(\statprob, \epsilon)
    \defeq \ceil{\frac{\log\frac{1}{\epsilon}}{|\log \singval|}
    + \frac{\log K}{|\log \singval|}} + 1
    ~~~ \mbox{and} ~~~
    \mixconst \defeq \frac{1}{|\log \singval|}.
  \end{equation*}
  For any $\epsilon \in \openleft{0}{1}$ and $c \in \R$,
  \begin{equation*}
    \statprob\left(\tmixtv(\statprob_{[t]}, \epsilon)
    \ge \tmixtv(\statprob, \epsilon) + \mixconst c\right)
    \le \exp(-c).
  \end{equation*}
\end{lemma}

We begin with the analysis of the random version of the Markov incremental
gradient descent (MIGD) procedure. As before, a token $\token(t)$ moves
among the processors in a network of $n$ nodes, but now
the transition matrix $\statprob$ governing the token is random.
At time $t$, the transition 
probability $\statprob(\token(t) = j \mid \token(t - 1) = i) = \statprob_{ij}(t)$,
where $\{\statprob(t)\}$ is an i.i.d.\ sequence of
doubly stochastic matrices. Let $\simplex_n$ denote the probability simplex in
$\R^n$ and $u(0) \in \simplex_n$ be arbitrary. Define the sequence $u(t + 1) =
\statprob(t) u(t)$, so $u(t)$ is the distribution of $\token(t)$
if the token has initial distribution $u(0)$. As shown by Boyd et
al.~\cite{BoydGhPrSh06} and further studied by Duchi et al.~\cite{DuchiAgWa12},
we obtain
\begin{equation}
  \E\left[\lone{u(t) - \onevec/n}\right]
  \le \sqrt{n} \ltwo{u(0)}^2 \lambda_2(\E[\statprob(1)^\top \statprob(1)])^t
  \le \sqrt{n} \lambda_2(\E[\statprob(1)^\top \statprob(1)])^t.
  \label{eqn:random-migd-matrices}
\end{equation}
Notably, with $\singval =
\lambda_2(\E[\statprob(1)^\top \statprob(1)]) < 1$ and $K = \sqrt{n}$,
 the estimate~\eqref{eqn:random-migd-matrices} satisfies the
conditions of Lemma~\ref{lemma:beta-mixing-to-probabilities},
 since
$\dtv(\statprob^t, \stationary) = \lone{u(t) - \onevec/n}$.
Generally, $\E[\statprob(1)^\top \statprob(1)]$ has much smaller second
eigenvalue than any of the random matrices $\statprob(t)$ (indeed, it may be
the case that $\lambda_2(\statprob(t)) = 1$ with probability 1, as in
randomized gossip~\cite{BoydGhPrSh06}).
Using~\eqref{eqn:random-migd-matrices}, if we define
$\lambda_2 = \lambda_2(\E[\statprob(1)^\top \statprob(1)])$, we may take
\begin{equation*}
  \tmixtv(\statprob, \epsilon) \le \frac{\log \frac{n}{\epsilon}}{
    1 - \lambda_2}
  ~~~ \mbox{and} ~~~
  \mixconst \le \frac{1}{1 - \lambda_2}
\end{equation*}
in Lemma~\ref{lemma:beta-mixing-to-probabilities}. Applying
Theorem~\ref{theorem:probabilistic-mixing} we obtain the following corollary.
\begin{corollary}
  \label{corollary:random-migd}
  Let the conditions of Theorem~\ref{theorem:probabilistic-mixing} hold, and
  in the notation of the previous paragraph, define $\lambda_2 \defeq
  \lambda_2(\E[\statprob(1)^\top \statprob(1)])$. Fix $\delta \in
  \openleft{0}{1}$. With stepsize choice $\stepsize(t) = \stepsize /
  \sqrt{t}$, there is a constant $C \le 4$ such that with
  probability at least $1 - \delta$
  \begin{align*}
    f(\what{x}(T)) - f(x^\star)
    & \le \inf_{\epsilon > 0} C \cdot
    \bigg(\frac{\radius^2}{\stepsize \sqrt{T}}
    + \frac{\log \frac{n}{\epsilon} + \log\frac{T}{\delta}}{
      1 - \lambda_2} \cdot \frac{\lipobj^2 \stepsize}{\sqrt{T}}
    + \epsilon \lipobj \radius \\
    & \qquad\qquad\quad ~ 
    + \frac{\lipobj \radius}{\sqrt{T}}
    \sqrt{\frac{\log\frac{Tn}{\epsilon \delta} \log(\frac{1}{\delta}
        \log\frac{Tn}{\epsilon \delta}
      / (1 - \lambda_2))}{1 - \lambda_2}}
    \bigg).
  \end{align*}
\end{corollary}

As an example of the applicability of this approach, suppose that in the
network of communicating agents used in MIGD, each communication link fails
with a probability $\gamma \in (0, 1)$, independently of the other links. Let
$\statprob$ denote the transition matrix used by the MIGD algorithm without
network failures. Then (under suitable conditions on
the network topology; see~\cite{DuchiAgWa12} for details)
\begin{equation*}
  \lambda_2(\E[\statprob(1)^\top \statprob(1)])
  \le \gamma + (1 - \gamma) \lambda_2(\statprob).
\end{equation*}
Applying Corollary~\ref{corollary:random-migd} and taking $\epsilon = 1/T$
and $\delta = 1/T^2$, we
obtain (ignoring doubly logarithmic factors) that there is a universal
constant $C$ such that with probability at least $1 - T^{-2}$
\begin{equation*}
  f(\what{x}(T)) - f(x^\star)
  \le C \cdot\left(
  \frac{\radius^2}{\stepsize \sqrt{T}}
  + \frac{\log(Tn)}{(1 - \gamma)(1 - \lambda_2(\statprob))}
  \cdot \frac{\lipobj^2 \stepsize}{\sqrt{T}}
  \right).
\end{equation*}
Roughly, we see the intuitive result that as the failure probability
$\gamma$ increases to 1, the convergence rate of the algorithm suffers;
for $\gamma$ bounded away from 1, we suffer only constant factor
losses over the rates in Corollary~\ref{corollary:migd}.

As another example of the applicability of
Theorem~\ref{theorem:probabilistic-mixing}, we look to problems where the
statistical sample space $\statsamplespace$ is uncountable. In such scenarios,
standard (finite-dimensional) Markov chain theory does not apply. Uncountable
spaces commonly arise, for example, in physical simulations of natural
phenomena or autoregressive processes~\cite{MeynTw09}, control
problems~\cite{KushnerYi03}, as well as in statistical learning applications,
such as Monte Carlo-sampling based variants of the expectation maximization
(EM) algorithm~\cite{WeiTa90}. To apply results based on
Assumption~\ref{assumption:uniform-mixing}, however, requires \emph{uniform
  ergodicity}~\cite[Chapter 16]{MeynTw09} of the Markov chain. Uniform
ergodicity is difficult to verify and often requires conditions essentially
equivalent to compactness of $\statsamplespace$.

Theorem~\ref{theorem:probabilistic-mixing} allows us to avoid such
difficulties. For concreteness, we focus on autoregressive moving average
(ARMA) processes, common models for control problems and
statistical time series. In general, an ARMA process is defined by
the recursion
\begin{equation}
  \label{eqn:arma-def}
  \statsample_{t + 1} = \arfunc(\statsample_t)
  + \arrandom(\statsample_t) W_t,
\end{equation}
where $\arfunc : \R^d \rightarrow \R^d$ and $\arrandom : \R^d \rightarrow
\R^{d \times d}$ are measurable, the innovations $W_t \in \R^d$ are
i.i.d.\ and $\cov(W_t)$ exists. When $\arfunc(z) = \arfunc
z$, that is, $\arfunc$ is identified with a matrix $\arfunc \in \R^{d \times
  d}$, and $\arrandom(z)$ is a constant matrix $\arrandom$, we recover the
standard linear ARMA model. The convergence of such processes is
area of recent research (e.g.,~\cite{MeynTw09,Mokkadem88,Liebscher05}),
but we focus particularly on the paper of Liebscher~\cite{Liebscher05}. As a
consequence of Liebscher's Theorem~2, we obtain that if $\arfunc(\statsample)
= \arfunc \statsample + h(\statsample)$, where $h(\statsample) =
o(\norm{\statsample})$ as $\norm{\statsample} \rightarrow \infty$, the matrix
$\arfunc$ satisfies $\singval_1(\arfunc) < 1$, and $\arrandom(\statsample)
\equiv \arrandom$ is a fixed matrix, then there exist constants $M \ge 0$ and
$\singval \in (0, 1)$ such that for all $t, \tau \in \N$
\begin{equation*}
  \E\left[\dtv(\statprob_{[t]}^{t + \tau}, \stationary)\right]
  \le M \singval^\tau
  ~~~ \mbox{whenever} ~~~
  \E[\norm{\statsample_0}] < \infty.
\end{equation*}
Here $\stationary$ is the stationary distribution of the ARMA
process~\eqref{eqn:arma-def}.

In particular, for any ARMA process~\eqref{eqn:arma-def} satisfying
the conditions, Lemma~\ref{lemma:beta-mixing-to-probabilities}
guarantees that Assumption~\ref{assumption:probabilistic-mixing}
holds.  We thus have the following corollary (it appears challenging
to obtain sharp constants~\cite{Liebscher05,MeynTw09}, so we leave
many unspecified).
\begin{corollary}
  \label{corollary:arma}
  Let the stochastic process $\statprob$ be the nonlinear ARMA process
  \begin{equation*}
    \statsample_{t + 1} = \arfunc \statsample_t + h(\statsample_t)
    + \arrandom W_t,
  \end{equation*}
  where the singular value $\singval_1(\arfunc) < 1$, $h(\statsample) =
  o(\norm{\statsample})$ as $\norm{\statsample} \rightarrow \infty$,
  and $\E[\ltwo{W_t}^2] < \infty$. Let
  Assumption~\ref{assumption:F-lipschitz} hold and $\delta \in (0, 1)$. Then
  there exist constants $M \ge 1$, $\singval \in (0, 1)$, and a universal
  constant $C \le 4$ such that with probability at least $1 - \delta$
  \begin{align*}
    f(\what{x}(T)) - f(x^\star) & \le \inf_{\epsilon > 0} C
    \cdot \bigg(\frac{\radius^2}{\stepsize \sqrt{T}}
    + \frac{\log\frac{MT}{\epsilon \delta}}{1 - \singval} \cdot
    \frac{\lipobj^2 \stepsize}{\sqrt{T}} + \epsilon \lipobj \radius \\
    & \qquad\qquad\qquad ~ + \frac{\lipobj \radius}{\sqrt{T}}
    \sqrt{\frac{\log \frac{MT}{\epsilon \delta}
        \log(\frac{1}{\delta}\log\frac{MT}{\epsilon \delta} / (1 - \singval))}{
        1 - \singval}}\bigg).
  \end{align*}
\end{corollary}
\noindent

Having provided Corollaries~\ref{corollary:random-migd}
and~\ref{corollary:arma}, we can now somewhat more concretely contrast our
results with those of Ram et al.~\cite{RamNeVe09a}. Ram et
al.'s results (essentially) apply when the set $\statsamplespace$ is finite,
as they define their objective $f(x) = \sum_{i=1}^n f_i(x)$ for functions
$f_i$; the ARMA example does not satisfy this property. In addition, Ram et
al.\ assume in the MIGD case that the network of agents $\{1, \ldots, n\}$ is
strongly connected over time: for any $t$, if one defines $E(t) = \{(i,
j) : \statprob(t)_{ij} > 0\}$, there exists a finite $t' \in \N$ such
that $\cup_{s = t}^{t'} E(s)$ defines a strongly connected graph. This
assumption need not hold for our analysis and fails for the examples
motivating Corollary~\ref{corollary:random-migd}.

\subsection{Slowly mixing processes}
\label{sec:slow-mixing}

Many ergodic processes do not enjoy the fast convergence rates of the previous
three examples. Thus we turn to a brief discussion of more slowly mixing
processes, which culminates in a result
(Corollary~\ref{corollary:asymptotics}) establishing asymptotic convergence of
EMD for any ergodic process satisfying
Assumption~\ref{assumption:uniform-mixing}.

\ifdefined\removepolynomial
\else

Our starting point is an example of a continuous state space Markov chain that
exhibits a mixing rate of the form (w.l.o.g.\ let $M \ge 1$ and $\beta \ge 0$)
\begin{equation}
  \tmixtv(\statprob, \epsilon) \le M\epsilon^{-\beta}.
  \label{eqn:subgeometric-mixing}
\end{equation}
Consider a Metropolis-Hastings sampler~\cite{RobertCa04}
with the stationary distribution $\stationary$, assumed (for simplicity) to
have a density $\stationarydensity$. The Metropolis-Hastings sampler uses a
Markov chain $Q$ as a ``proposal'' distribution, where $Q(\statsample_t,
\cdot)$ denotes the distribution of $\statsample_{t+1}$ conditioned on
$\statsample_t$, and $Q(\statsample_t, \cdot)$ is assumed to have density
$q(\statsample_t, \cdot)$. The Markov chain constructed from $Q$ and
$\stationary$ transitions from a point $\statsample_1$ to $\statsample_2$ as
follows: first, the procedure samples $\statsample$ according to
$Q(\statsample_1, \cdot)$; second, the sample is accepted and $\statsample_2$
is set to $\statsample$ with probability
$\min\{\frac{\stationarydensity(\statsample_2) q(\statsample_2,
  \statsample_1)}{\stationarydensity(\statsample_1)
  q(\statsample_1,\statsample_2)}, 1\}$, otherwise $\statsample_2 =
\statsample_1$. Metropolis-Hastings algorithms are the backbone for a large
family of MCMC sampling procedures~\cite{RobertCa04}. When $Q$
generates independent samples---that is, $q(\statsample, \cdot) \equiv
q(\cdot)$ for all $\statsample$---then the associated Markov chain is
uniformly ergodic only when the ratio
$q(\statsample)/\stationarydensity(\statsample)$ is bounded away from
zero over the sample space $\statsamplespace$~\cite[Chapter 20]{MeynTw09}.

When such a lower bound fails to exist, Metropolis-Hastings has slower mixing
times. Jarner and Roberts~\cite{JarnerRo02} give an example where
$\stationary$ is uniform on $[0,1]$ and the density $q(x) = (r+1)x^r$ for some
$r > 0$; they show in this case a polynomial mixing
rate~\eqref{eqn:subgeometric-mixing} with $\beta = 1/r$; other examples of
similar rates include particular random walks on $[0,\infty)$ or queuing
  processes in continuous time.

We now state a corollary of our main results when the mixing time takes the
form~\eqref{eqn:subgeometric-mixing}.
\begin{corollary}
  \label{corollary:subgeometric-convergence}
  Let $x(t)$ evolve according to the EMD update~\eqref{eqn:ergodic-md},
  where the sampling distribution $\statprob$ is a
  Markov chain with $\tmixtv(\statprob, \epsilon) \le M \epsilon^{-\beta}$.
  Assume that $T \ge (\radius / \lipobj)^2$.
  Under Assumption~\ref{assumption:F-lipschitz} and with $\stepsize(t) \equiv
  \frac{R}{\lipobj}T^{-(\beta+1)/(\beta+2)}$,
  \begin{equation*}
    \E[f(\what{x}(T)]) - f(x^\star) \le
    \frac{5\lipobj \radius M^{\frac{1}{\beta + 1}}}{T^{\frac{1}{2 + \beta}}}.
  \end{equation*}
  The stepsize choice $\stepsize(t) = \radius / (\lipobj \sqrt{t})$
  gives that
  \begin{equation*}
    \E[f(\what{x}(T))] - f(x^\star)
    \le \frac{3 \lipobj \radius}{2 \sqrt{T}}
    + \frac{e \lipobj \radius (3M/2)^{\frac{1}{
          \beta + 1}}}{T^{\frac{1}{2 + 2 \beta}}}.
  \end{equation*}
\end{corollary}
\begin{proof}
  By applying the bound in Corollary~\ref{corollary:expected-convergence},
  we see that the expected convergence
  rate for the fixed setting of $\stepsize(t) \equiv \stepsize$ in
  the statement of the corollary is
  \begin{equation*}
    \frac{\radius^2}{2 T \stepsize}
    + \frac{\lipobj^2}{2} \stepsize
    + 2 \epsilon \lipobj \radius
    + 2 M \epsilon^{-\beta} \lipobj^2 \stepsize
    + \frac{M \epsilon^{-\beta} \radius \lipobj}{T}
    \le
    \frac{\radius^2}{2 T \stepsize} + \frac{\lipobj^2}{2} \stepsize + 2\epsilon
    \lipobj \radius + 3 M \epsilon^{-\beta} \lipobj^2 \stepsize,
  \end{equation*}
  using the assumption that $T \ge (\radius / \lipobj)^2$. We can choose
  $\epsilon$ arbitrarily, so set $\epsilon = (3 \beta \lipobj M \stepsize /
  \radius)^{1 / (1 + \beta)}$. Using the proposed stepsize $\stepsize(t) =
  (\radius / \lipobj)T^{-(\beta+1)/(\beta+2)}$, we find that the above is
  equal to
  \begin{equation*}
    \frac{\radius^2}{2 T \stepsize}
    + \frac{\lipobj^2}{2} \stepsize
    + \big(2 + \beta^{-\frac{\beta}{1 + \beta}}\big)
    \stepsize^{\frac{1}{1 + \beta}}(3M)^{\frac{1}{1 + \beta}}
    \lipobj^{\frac{2 + \beta}{1 + \beta}} R^{\frac{\beta}{1 + \beta}}
    \le
    \frac{\lipobj\radius}{2 T^{\frac{1}{2 + \beta}}}
    + \frac{\lipobj \radius}{2 T^{\frac{\beta + 1}{\beta + 2}}}
    + \frac{4\lipobj \radius (3 M)^{\frac{1}{1 + \beta}}}{
      T^{\frac{1}{2 + \beta}}},
  \end{equation*}
  where we use $2 + \beta^{-\beta / (1 + \beta)} \le 4$.
  Noting that $\beta \ge 0$ yields the first statement.

  With the step size choice $\stepsize(t) = \stepsize / \sqrt{t}$ with
  multiplier $\stepsize = \radius / \lipobj$, we can apply
  Theorem~\ref{theorem:expected-convergence}, along with the
  bound~\eqref{eqn:sqrt-integral-bound} in the proof of
  Corollary~\ref{corollary:geometric-convergence}, to see that
  \begin{align}
    \E[f(\what{x}(T))] - f(x^\star)
    & \le \frac{3\radius \lipobj}{2 \sqrt{T}}
    + 2 \epsilon \lipobj \radius + \frac{2 \tmixtv(\statprob, \epsilon)
      \lipobj \radius}{\sqrt{T}}
    + \frac{\tmixtv(\statprob, \epsilon) \lipobj \radius}{T}
    \label{eqn:slow-mixing-ez-bound}
  \end{align}
  Noting that $1 / T + 2 / \sqrt{T} \le 3 / \sqrt{T}$, we turn to bounding
  \begin{equation}
    2 \epsilon \lipobj \radius
    + \frac{3 \tmixtv(\statprob, \epsilon) \lipobj \radius}{\sqrt{T}}
    \le 2\epsilon \lipobj \radius + \epsilon^{-\beta}
    \frac{3 M \lipobj \radius}{\sqrt{T}}.
    \label{eqn:slow-mixing-eps-control}
  \end{equation}
  Since $\epsilon$ does not enter into the algorithm at all, we are free to
  minimize over $\epsilon$, and taking derivatives we see that we must
  solve
  \begin{equation*}
    2\lipobj \radius - \beta \epsilon^{-\beta - 1}
    \frac{3 M \lipobj\radius}{\sqrt{T}}
    = 0
    ~~~ \mbox{or} ~~~
    \epsilon = \left(\frac{3 M \beta}{2\sqrt{T}}\right)^{\frac{1}{\beta + 1}}.
  \end{equation*}
  Since $\beta^{1 / (\beta + 1)} \le e/2$ and $\beta^{-\beta / (\beta +
    1)} \le e/2$, this choice of $\epsilon$ in the
  bound~\eqref{eqn:slow-mixing-eps-control} yields
  \begin{equation*}
    \inf_{\epsilon} \left\{
    2 \epsilon \lipobj \radius + \frac{3 \tmixtv(\statprob, \epsilon)
      \lipobj\radius}{\sqrt{T}}\right\}
    \le e \lipobj \radius (3M/2)^{\frac{1}{\beta + 1}}
    \cdot T^{\frac{-1}{2 \beta + 2}}.
  \end{equation*}
  By inspection, this inequality and the convergence
  guarantee~\eqref{eqn:slow-mixing-ez-bound} give the second statement of
  the corollary.  
\end{proof}
\fi

\ifdefined\removepolynomial %

In general, attaining optimal rates of convergence for slowly mixing processes
requires knowledge of the mixing rate of $\statprob$. Choosing an incorrect
rate---that is, setting $\stepsize(t) \propto t^{-m}$ for incorrect $m$---can
lead to substantially slower convergence. In contrast, as noted in
\S~\ref{sec:main-results}, our other bounds are robust to mis-specification of
the step size so long as the ergodic process $\statprob$ mixes suitably
quickly and we can choose $\stepsize(t) \propto t^{-1/2}$. There is a simple
technique we can use to demonstrate that the stepsize choice $\stepsize(t) =
\stepsize/\sqrt{t}$ provably yields convergence, both in expectation and with
high probability, even for slowly mixing processes.
\else

A weakness of the above bound is that the sharper rate of convergence requires
knowledge of the mixing rate of $\statprob$, and choosing the polynomial
incorrectly can lead to significantly slower convergence.  In contrast, as
noted in \S~\ref{sec:main-results}, our other bounds are robust to
mis-specification of the step size so long as the ergodic process $\statprob$
mixes suitably quickly and we can choose $\stepsize(t) \propto
t^{-1/2}$. Nonetheless, Corollary~\ref{corollary:subgeometric-convergence}
gives a finite sample convergence rate whose dependence on the slower mixing
of the ergodic process is clear. In addition, the proof of
Corollary~\ref{corollary:subgeometric-convergence} exhibits a simple technique
we can use to demonstrate that the stepsize choice $\stepsize(t) = \stepsize /
\sqrt{t}$ provably yields convergence, both in expectation and with high
probability.
\fi
To be specific, note that the bound in
Corollary~\ref{corollary:expected-convergence} guarantees that for
$\what{x}(T) = \frac{1}{T} \sum_{t=1}^T x(t)$, if we choose $\stepsize(t) =
\stepsize / \sqrt{t}$ then
\begin{equation}
  \label{eqn:bound-with-epsilon}
  \E[f(\what{x}(T))] - f(x^\star) \le \frac{\radius^2}{2 \stepsize \sqrt{T}}
  + \frac{\lipobj^2 \stepsize}{\sqrt{T}} + 3 \epsilon \lipobj \radius
  + \frac{2 \tmix(\statprob, \epsilon) \lipobj^2 \stepsize}{\sqrt{T}}
  + \frac{\tmix(\statprob, \epsilon) \radius \lipobj}{T},
\end{equation}
where $\tmix$ denotes either the Hellinger or total variation mixing time.
The convergence guarantee~\eqref{eqn:bound-with-epsilon} holds
regardless of our choice of $\epsilon$, so we can choose $\epsilon$ minimizing
the right-hand side. That is (setting $\stepsize = \radius / \lipobj$ for notational
convenience),
\begin{equation*}
  \E[f(\what{x}(T))] - f(x^\star)
  \le \frac{3 \lipobj\radius}{2\sqrt{T}} + \inf_{\epsilon \ge 0}
  \left\{3 \epsilon \lipobj \radius
  + \frac{2 \tmix(\statprob, \epsilon) \lipobj \radius}{
    \sqrt{T}} + \frac{\tmix(\statprob, \epsilon) \lipobj\radius}{T}\right\}.
\end{equation*}
For any fixed $\epsilon > 0$, the term inside the infimum 
decreases to $4 \epsilon \lipobj \radius$ as $T \uparrow \infty$, so the
infimal term decreases to zero as $T \uparrow \infty$. High probability
convergence follows similarly by using
Theorem~\ref{theorem:highprob-convergence}, since for
any $\delta_T > 0$ we have
\begin{align}
  f(\what{x}(T)) - f(x^\star)
  & \le \frac{3 \lipobj \radius}{2 \sqrt{T}} + \inf_{\epsilon \ge 0}
  \bigg\{\epsilon \lipobj \radius
  + \frac{2 \tmixtv(\statprob, \epsilon) \lipobj \radius}{
    \sqrt{T}} + \frac{\tmixtv(\statprob, \epsilon) \lipobj\radius}{T}
  \nonumber \\
  & \qquad\qquad\qquad\qquad ~
  + \frac{4\lipobj \radius}{\sqrt{T}} \sqrt{
    \tmixtv(\statprob, \epsilon)
    \log \frac{\tmixtv(\statprob, \epsilon)}{\delta}}
  \bigg\}
  \label{eqn:slow-mixing-highprob}
\end{align}
with probability at least $1 - \delta_T$.
We obtain the following corollary:
\begin{corollary}
  \label{corollary:asymptotics}
  Define $\what{x}(T) = \frac{1}{T} \sum_{t=1}^T x(t)$.  Under the conditions
  of Theorem~\ref{theorem:highprob-convergence}, the stepsize sequence
  $\stepsize(t) = \stepsize / \sqrt{t}$ for any $\stepsize > 0$ yields
  $f(\what{x}(T)) \rightarrow f(x^\star)$ as $T \rightarrow \infty$ both in
  expectation and with probability 1.
\end{corollary}
\begin{proof}
  Fix $\gamma > 0$ and let $\event{T}$ denote the event that $f(\what{x}(T)) -
  f(x^\star) > \gamma$.  We use the Borel-Cantelli lemma~\cite{Billingsley86}
  to argue that $\event{T}$ occurs for only a finite number of $T$ with
  probability one.  Take the sequence $\delta_T = 1/T^2$ (any sequence for
  which $\log(1 / \delta_T) / T \downarrow 0$ as $T \rightarrow \infty$ and
  $\sum_{T=1}^\infty \delta_T < \infty$ will suffice) and choose some $T_0$
  such that the right-hand side of the bound~\eqref{eqn:slow-mixing-highprob}
  is less than $\gamma$. Then we have
  \begin{equation*}
    \sum_{T=1}^\infty \statprob(f(\what{x}(T)) - f(x^\star) > \gamma)
    = \sum_{T=1}^\infty \statprob(\event{T})
    \le T_0 + \sum_{T = T_0 + 1}^\infty \statprob(\event{T})
    \le T_0 + \sum_{T = 1}^\infty \delta_T
    < \infty.
  \end{equation*}
  For any $\gamma > 0$, we have $\statprob(f(\what{x}(T)) -
  f(x^\star) > \gamma ~ \mbox{i.o.}) = 0$.
\end{proof}

%% file: experiments.tex
\section{Numerical results}
\label{sec:experiments}

In this section, we present simulation experiments that further
investigate the behavior of the EMD
algorithm~\eqref{eqn:ergodic-md}. Though
Theorem~\ref{theorem:lower-bound} guarantees that our rates are
essentially unimprovable, it is interesting to compare our method with
other natural well-known procedures. We would also like to understand
the benefits of the mirror descent approach for problems in which the
natural geometry is non-Euclidean as well as the robustness properties
of the algorithm.

\subsection{Sampling strategies}

\newcommand{\normal}{\ensuremath{N}}

For our first experiment, we study the performance of the EMD
algorithm on a robust system identification task~\cite{PolyakTs80},
where we assume the data is generated by an autoregressive process.
More precisely, our data generation mechanism is as follows. For each
experiment, we set the matrix $\arfunc$ to be a sub-diagonal matrix
(all entries are 0 except those on the sub-diagonal), where
$\arfunc_{i,i-1}$ is drawn uniformly from $[.8, .99]$.  We then draw a
vector $u$ uniformly from surface of the $d$-dimensional $\ell_2$-ball
of radius $\radius = 5$. The data comes in pairs $(\statsample_t^1,
\statsample_t^2) \in \R^d \times \R$ with $d = 50$ and is generated as
follows:
\begin{equation}
  \label{eqn:ar-system-id-model}
  \statsample_t^1 = \arfunc \statsample_{t-1}^1 + e_1 W_t,
  ~~~ \statsample_t^2 = \<u, \statsample_t^1\> + E_t,
\end{equation}
where $e_1$ is the first standard basis vector, $W_t$ are i.i.d.\ samples from
$\normal(0, 1)$, and $E_t$ are i.i.d.\ bi-exponential random variables with
variance 1.
Polyak and Tsypkin~\cite{PolyakTs80} suggest
the method of least-moduli for the system identification task, setting
\begin{equation*}
  F(x; (\statsample^1, \statsample^2))
  = \left|\<x, \statsample^1\> - \statsample^2 \right|,
\end{equation*}
which is optimal (in a minimax sense) when little is known about the noise
distribution~\cite{PolyakTs80}. Our minimization problem is
\begin{equation}
  \label{eqn:least-median}
  \minimize_x ~ f(x) =
  \E_\stationary\left[\left|\<x, \statsample^1\> -
    \statsample^2\right|\right] ~~ \subjectto ~~ \ltwo{x} \le \radius,
\end{equation}
where $\stationary$ is the stationary distribution of the AR
model~\eqref{eqn:ar-system-id-model} and we take $\radius = 5$.

We use this experiment to investigate two issues. In addition to studying the
performance of the EMD algorithm in minimizing the expected
objective~\eqref{eqn:least-median}, we compare EMD to a natural alternative.
In many engineering applications it is possible to
generate samples from a distribution $\statprob$ that converges to
$\stationary$, in which case a natural algorithm is to use the so-called
``multiple replications'' approach (e.g.,~\cite{GelmanRu92}). In this
approach, one specifies initial conditions of the stochastic process
$\statprob$, then simulates it for some number $k$ of steps, and obtains a
sample $\statsample$ according to $\statprob^k$, which (hopefully) is close to
$\stationary$. Repeating this, one can obtain multiple independent samples
$\statsample$ from $\statprob^k$, then use standard algorithms and analyses
for independent data.\footnote{This approach is inapplicable when the data
  $\statsample_t$ comes from a real (unsimulated) source,
  such as in streaming, online optimization, or statistical applications,
  though the EMD algorithm still applies.}  A difficulty with this
approach---which we see in our experiments---is that the mixing time of the
process $\statprob$ may be unknown, and if $\statprob^k$ does not converge
precisely to $\stationary$ for any finite $k \in \N$, then any algorithm using
such samples will be biased even in the limit of infinite gradient steps.

\begin{figure}
  \begin{center}
    \begin{tabular}{cc}
      \includegraphics[width=.47\columnwidth]{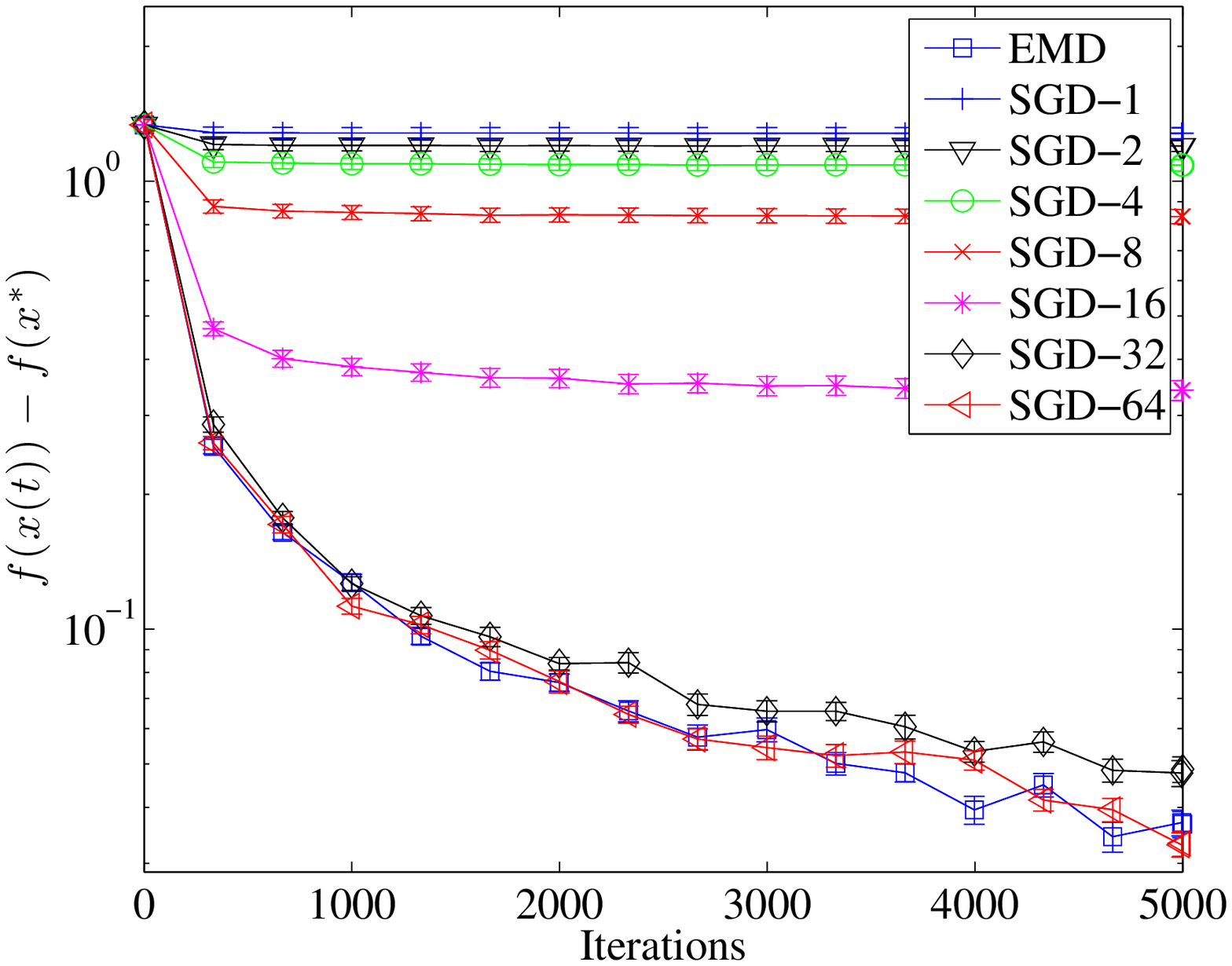} &
      \includegraphics[width=.47\columnwidth]{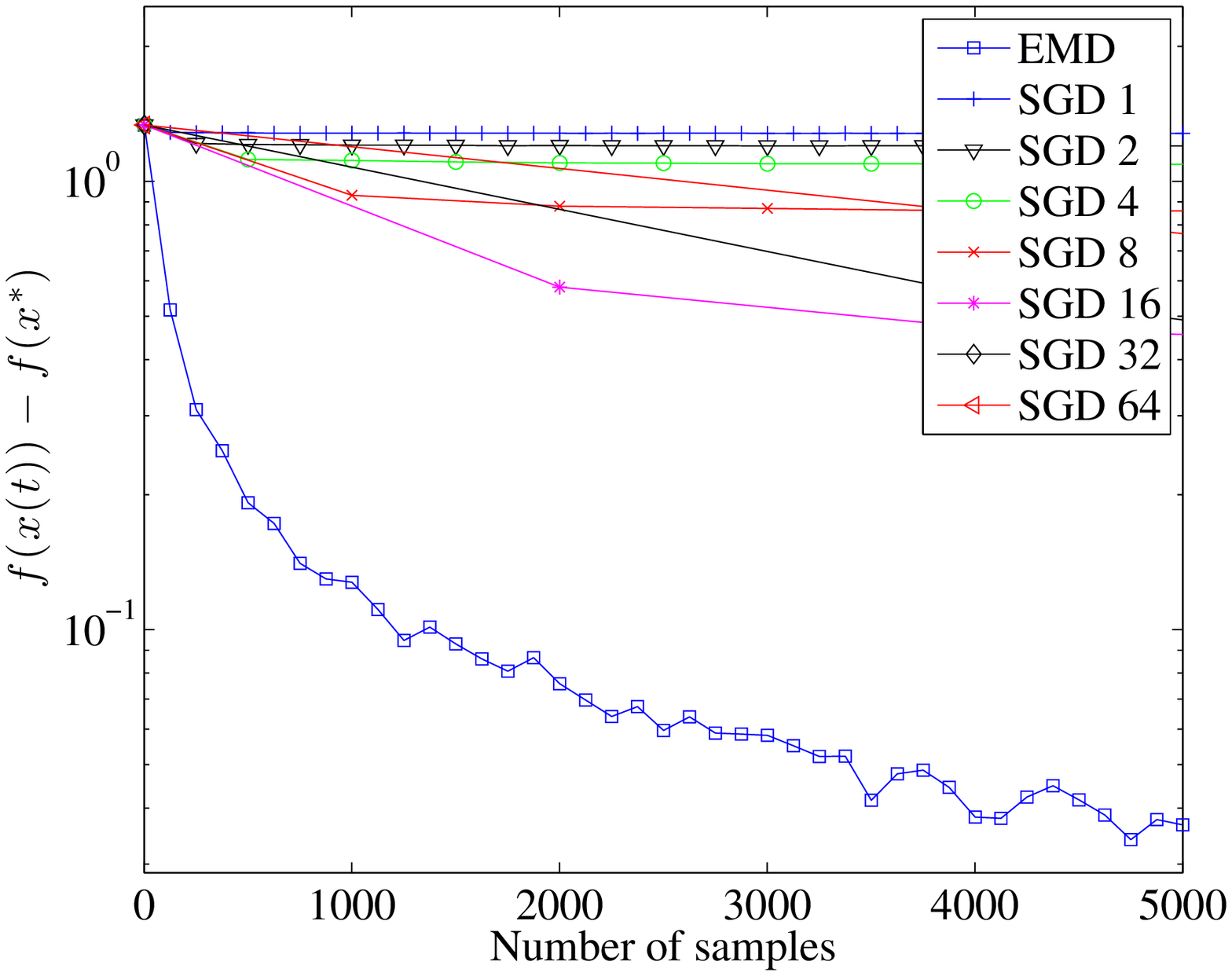}
    \end{tabular}
    \caption{\label{fig:ar-system-id}
      Performance of the EMD algorithm~\eqref{eqn:ergodic-md} on a
      robust system identification task where data is generated according to
      an autoregressive process.
    }
  \end{center}
\end{figure}

As a natural representative from the multiple-replications family of
algorithms, we use the classical stochastic gradient descent (SGD) algorithm
(in the form studied by Nemirovski et al.~\cite{NemirovskiJuLaSh09}). To
generate each sample for SGD, we begin with the point $\statsample_1^1 = 0$
and perform $k$ of steps of the
procedure~\eqref{eqn:ar-system-id-model}, using $\statsample_k^{\{1,2\}}$ to
compute subgradients for SGD. For EMD, we use the proximal function $\prox(x)
= \half \ltwo{x}^2$, which yields the direct analogue of stochastic gradient
descent.  To measure the objective value $f(x)$, we generate an independent fixed
sample of size $N = 10^5$ from the process~\eqref{eqn:ar-system-id-model},
using $f(x) \approx \frac{1}{N} \sum_{i=1}^N F(x; \statsample_i)$. For each
algorithm, to choose the stepsize multiplier $\stepsize \propto \radius /
\lipobj$, we estimate $\lipobj$ by taking 100 samples $\statsample_t^1$ and
computing the empirical average of $\ltwos{\statsample_t^1}^2$. For EMD, we
deliberately underestimate the mixing time by the constant $1$ (other
estimates of the mixing time yielded similar performance).

In Figure~\ref{fig:ar-system-id}, we show the convergence behavior (as a
function of number of samples) for the EMD algorithm compared with the
behavior of the stochastic gradient method for different numbers $k$ of
initial simulation steps before obtaining the sample $\statsample$ used in
each iteration of SGD. The line in each plot corresponding to SGD-$k$ shows
the convergence of stochastic gradient descent as a function of number of
iterations when $k$ initial samples are used for each independent sample
$\statsample$. The left plot in Figure~\ref{fig:ar-system-id} makes clear that
if the mixing time is underestimated, the multiple-replications approach
fails. As demonstrated by our theory, however, EMD still guarantees
convergence even with poor stepsize choices (see also our experiments in the
next section). For large enough mixing time estimate $k$, the
multiple-replication stochastic gradient method and the EMD method have
comparable performance in terms of optimization error as a function of number
of gradient steps. The right plot in Figure~\ref{fig:ar-system-id} shows the
convergence behavior of the competing methods as a function of the number of
samples of the stochastic process~\eqref{eqn:ar-system-id-model}. From this
plot, it becomes clear that using each sample sequentially as in EMD---rather
than attempting to draw independent samples at each iteration---is the more
computationally efficient approach.

\subsection{Robustness and non-Euclidean geometry}

\begin{figure}
  \begin{center}
    \begin{tabular}{cc}
      \includegraphics[width=.47\columnwidth]{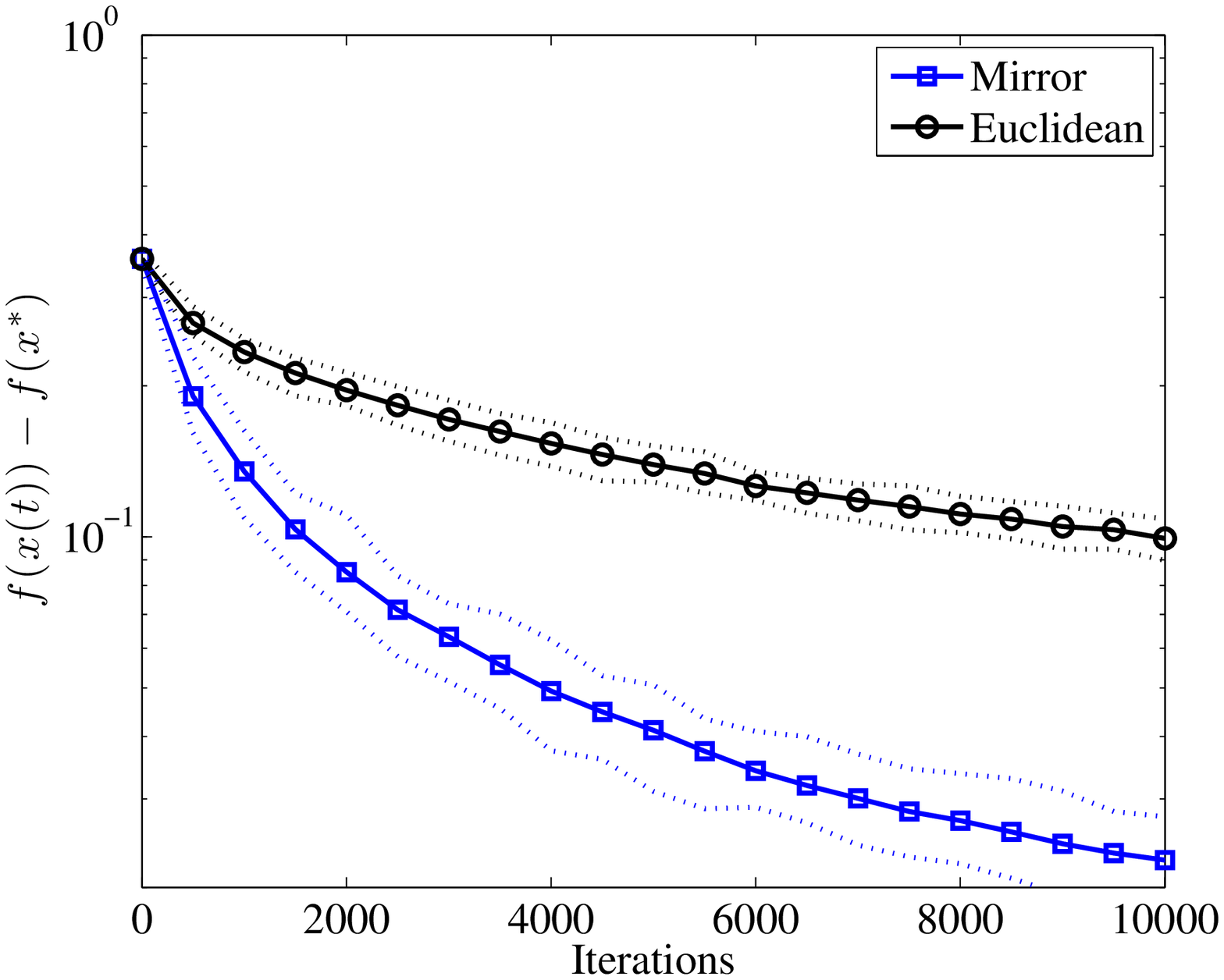} &
      \includegraphics[width=.47\columnwidth]{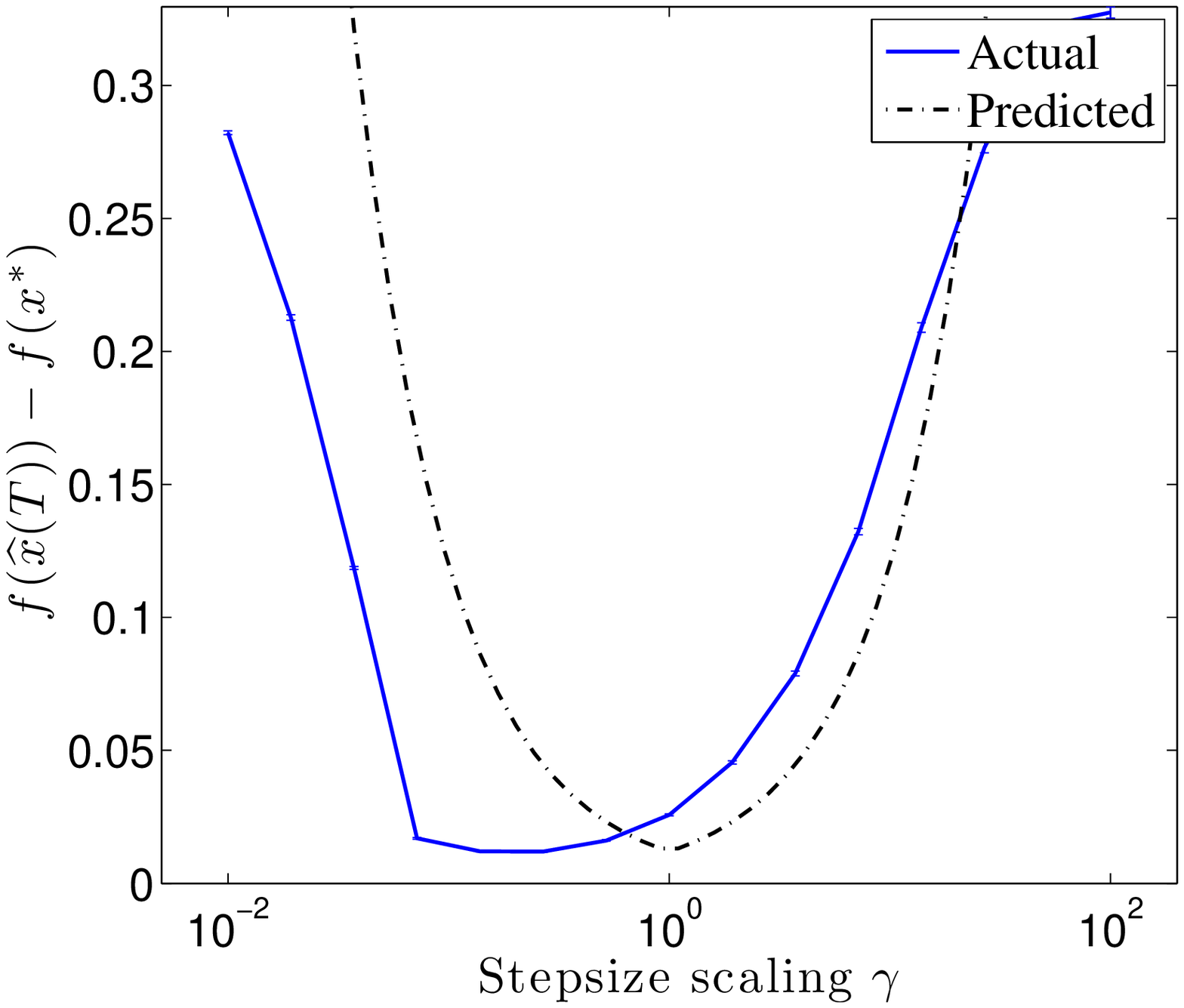}
    \end{tabular}
    \caption{\label{fig:migd} Left: optimization error on a statistical
      machine learning task of the Euclidean variant of the EMD
      algorithm~\eqref{eqn:ergodic-md} versus that of the $\ell_p$-norm
      variant with $\prox(x) = \half \norm{x}_p^2$, $p = 1 + 1 / \log d$,
      plotted against number of iterations. Right: robustness of the EMD
      algorithm~\eqref{eqn:ergodic-md} to modifications in the choice of
      stepsize.}
  \end{center}
\end{figure}

In our second numerical experiment, we study an important problem that takes
motivation from distributed statistical machine learning
problems: the support vector
machine problem~\cite{CortesVa95}, where the samples $\statsample \in \R^d$
and the instantaneous objective is
\begin{equation*}
  F(x; \statsample) = \hinge{1 - \<\statsample, x\>}.
\end{equation*}
We study the performance of the EMD algorithm for the distributed Markov
incremental mirror descent framework in \S~\ref{sec:migd}. In the
notation of \S~\ref{sec:migd}, we simulate $n = 50$ ``processors,'' and
for each we draw a sample of $m = 50$ samples according to the following
process. Before performing any sampling, we set $u$ to be a random vector from
$\{x \in \R^d : \lone{x} \le \radius\}$, where $\radius = 5$ and $d =500$. To
generate the $i$th data sample, we draw a vector $a_i \in \R^d$ with entries
$a_{i,j} \in \{-1, 1\}$ each with probability $\half$, and set $b_i =
\sign(\<a_i, u\>)$. With probability $.05$, we flip the sign of $b_i$ (this
makes the problem slightly more difficult, as no vector $x$ will perfectly
satisfy $b_i = \sign(\<a_i, x\>)$), and regardless we set $\statsample_i = b_i
a_i$. We thus generate a total of $N = nm = 2500$ samples, and set the $i$th
objective in the distributed minimization problem~\eqref{eqn:markov-problem}
to be
\begin{equation}
  \label{eqn:processor-objective}
  f_i(x)
  = \frac{1}{m} \sum_{k = m(i - 1) + 1}^{mi} F(x; \statsample_k)
  = \E_{\stationary_i}[F(x; \statsample)]
  = \E_{\stationary_i}\left[\hinge{1 - \<\statsample, x\>}\right],
\end{equation}
where $\stationary_i$ denotes the uniform distribution over the $i$th block of
$m$ samples. Our algorithm to minimize $f(x) = \frac{1}{n} \sum_{i=1}^n
f_i(x)$ is the Markov analogue~\eqref{eqn:general-markov-algorithm} of the
general EMD algorithm~\eqref{eqn:ergodic-md}. We minimize $f(x)$ over $\{x :
\lone{x} \le \radius\}$ offline using standard LP software to obtain the
optimal value $f(x^\star)$ of the problem.

We use the objectives~\eqref{eqn:processor-objective} to (i) understand the
effectiveness of allowing non-Euclidean proximal functions $\prox$ in the
update~\eqref{eqn:ergodic-md} and (ii) study the robustness of the EMD
algorithm~\eqref{eqn:ergodic-md} to stepsize selection.  We begin with the
first goal. As noted by Ben-Tal et al.~\cite{Ben-TalMaNe01}, the choice
$\prox(x) = \half \norm{x}_p^2$ with $p = 1 + 1 / \log(d)$ yields a nearly
optimal dependence on dimension in non-Euclidean gradient methods. Let $\tmix$
denote the mixing time of the Markov chain (for Hellinger or total variation
distance). Applying Corollary~\ref{corollary:migd} and the analysis of Ben-Tal
et al.\ with this choice of proximal function and $\stepsize = \radius /
\sqrt{\log(d) \tmix}$ yields
\begin{equation*}
  \E[f(\what{x}(T))] - \inf_{x \in \xdomain} f(x)
  = \order\left(\frac{\radius \sqrt{ \tmix \log d}}{\sqrt{T}}\right),
\end{equation*}
since $\linf{\partial_x F(x; \statsample)} \le \linf{\statsample} = 1$ by
our sampling of the vectors $a_i \in \{-1, 1\}^d$, and $\radius$ is the
radius of $\xdomain$ in $\ell_1$-norm. Compared to the Euclidean
variant~\cite{JohanssonRaJo09,RamNeVe09a} with $\prox(x) = \half \ltwo{x}^2$,
whose convergence rate also follows from Corollary~\ref{corollary:migd}, this
is an improvement of $\sqrt{d / \log d}$, since $\ltwo{\partial_x F(x;
  \statsample)}$ can be as large as $\sqrt{d}$.

We plot the results of 50 simulations of the distributed minimization problem
in the left plot of Figure~\ref{fig:migd}. For our underlying network
topology, we use a $4$-connected cycle (each node in the cycle is connected to
its $4$ neighbors on the right and left) and $n = 50$ nodes. The line of blue
squares is the mirror-descent approach with $\prox(x) = \half \norm{x}_p^2$
with $p = 1 + 1 / \log(d)$ (we use $d = 500$), while the black line of circles
denotes the Euclidean variant with $\prox(x) = \half \ltwo{x}^2$. The dotted
lines below and above each plot give the $5$th and $95$th percentiles,
respectively, of the optimization error across all simulations. For each
algorithm, we use the optimal step size setting $\stepsize(t)$ predicted by
our theory (recall Corollary~\ref{corollary:migd}).  It is clear that the
non-Euclidean variant enjoys better performance, as our theory (and previous
work on the dimension dependence of mirror
descent~\cite{NemirovskiYu83,NemirovskiJuLaSh09,Ben-TalMaNe01,BeckTe03})
suggests.

The final simulation we perform is on the same problem, but we investigate the
robustness of the EMD algorithm to mis-specified stepsizes.  We take the
stepsize $\stepsize^*$ predicted by our theory
(Corollary~\ref{corollary:migd}), and use $\stepsize(t) = \gamma \stepsize^* /
\sqrt{t}$ for values of $\gamma$ uniformly logarithmically spaced from $\gamma
= 10^{-2}$ to $\gamma = 10^2$. The plot on the right side of
Figure~\ref{fig:migd} shows the mean optimality gap of $\what{x}(T)$ after $T
= 10000$ iterations for different values of $\gamma$, along with standard
deviations, across 50 experiments. The black dotted line shows the predicted
optimality gap as a function of the mis-specification (recall our discussion
on robustness following Corollary~\ref{corollary:geometric-convergence}). The
EMD algorithm is certainly affected by mis-specification of the initial
stepsize, though for a range of values of roughly $\gamma = 10^{-1}$ to
$\gamma = 10$, the performance degradation does not appear extraordinary.  In
addition, our experiments show that our theoretical predictions appear
to capture the empirical behavior of the method quite well.

%% file: analysis.tex
\section{Analysis}
\label{sec:analysis}

In this section, we analyze the convergence of the EMD algorithm from
Section~\ref{sec:algorithm}. Our first subsection lays the groundwork,
gives necessary notation, and provides a few optimization-based
results.  The second subsection contains the proofs of results on
expected rates of convergence, while the third subsection shows how to
achieve convergence guarantees with high probability. The fourth
subsection shows the convergence of the EMD method under probabilistic
(random) mixing times, while the final subsection proves the
order-optimality of the EMD method.

\subsection{Definitions, assumptions, and optimization-based results}
\label{sec:assumption-consequences}

To state our results formally, we begin by giving a few standard
definitions and collecting a few consequences of
Assumptions~\ref{assumption:one-step-variance}
and~\ref{assumption:F-lipschitz} that make our proofs cleaner.  Recall
the measurable selection $\gradfunc$, where $\gradfunc(x; \statsample)
\in \partial_x F(x; \statsample)$ represents a fixed and measurable
element of the subgradient of $F(\cdot; \statsample)$ evaluated at
$x$, and the EMD algorithm~\eqref{eqn:ergodic-md} has $g(t) =
\gradfunc(x(t); \statsample_t)$. By our assumptions on $F$, for any
distribution $Q$ for which the expectations below are defined,
expectation and subdifferentiation
commute~\cite{RockafellarWe82,Bertsekas73}:
\begin{equation*}
  f_Q(x) \defeq \E_Q[F(x; \statsample)]
  = \int_{\statsamplespace} F(x; \statsample) dQ(\statsample)
  ~~~ \mbox{then} ~~~
  \partial f_Q(x) = \E_Q[\partial F(x; \statsample)].
\end{equation*}
In particular, $\E_\stationary[\partial F(x; \statsample)] = \partial f(x)$
and $\E_\stationary[\gradfunc(x; \statsample)] \in \partial f(x)$.  In
addition, the compactness assumption that \mbox{$\divergence(x^\star, x(t)) \le
  \half \radius^2$} for all $t$ coupled with the strong convexity of $\prox$
implies
\begin{equation}
  \label{eqn:consequence-compactness}
  \norm{x(t) - x^\star}^2
  \le 2\divergence(x^\star, x(t))
  \le \radius^2
  ~~~ \mbox{so} ~~~
  \norm{x(t) - x^\star} \le \radius.
\end{equation}

We now provide two relatively standard
optimization-theoretic results that make our proofs substantially
easier. To make the presentation self-contained, we give proofs of
these results in Appendix~\ref{appendix:md-proofs}.  The two lemmas
are essentially present in earlier
work~\cite{NemirovskiYu83,BeckTe03}, but our stochastic setting
requires a bit of care.
\begin{lemma}
  \label{lemma:regret-bound}
  Let $x(t)$ be defined by the EMD
  update~\eqref{eqn:ergodic-md}. For any $\tau \in \N$ and any $x^\star \in
  \xdomain$,
  \begin{equation*}
    \sum_{t = \tau + 1}^T F(x(t); \statsample_t)
    - F(x^\star; \statsample_t)
    \le \frac{1}{2 \stepsize(T)} \radius^2
    + \sum_{t = \tau + 1}^T \frac{\stepsize(t)}{2} \dnorm{g(t)}^2.
  \end{equation*}
\end{lemma}
\begin{lemma}
  \label{lemma:xt-diff}
  Let $x(t)$ be generated according to the EMD
  algorithm~\eqref{eqn:ergodic-md}. Then
  \begin{equation*}
    \norm{x(t) - x(t + 1)} \le \stepsize(t) \dnorm{g(t)}.
  \end{equation*}
\end{lemma}

\subsection{Expected convergence rates}
\label{sec:proof-of-theorem-expected}

Now that we have established the relevant optimization-based results
and setup in Section~\ref{sec:assumption-consequences}, the proof of
Theorem~\ref{theorem:expected-convergence} requires that we understand
the impact of the ergodic sequence $\statsample_1, \statsample_2,
\ldots$ on the EMD procedure. The key equality that allows us to prove
Theorems~\ref{theorem:expected-convergence}
and~\ref{theorem:highprob-convergence} is the following: for any $\tau
\ge 0$,
\begin{align}
  \sum_{t = 1}^T f(x(t)) - f(x^\star) & = \sum_{t = 1}^{T - \tau}
  f(x(t)) - f(\opt) - F(x(t); \statsample_{t + \tau}) + F(\opt;
  \statsample_{t + \tau}) \nonumber \\
  & \quad ~ + \sum_{t = 1}^{T - \tau}
  F(x(t); \statsample_{t + \tau}) - F(x(t + \tau); \statsample_{t +
    \tau})
  \label{eqn:expected-function-to-regret} \\
  & \quad ~ + \sum_{t = \tau + 1}^T F(x(t); \statsample_t) - F(x^\star;
  \statsample_t) + \sum_{t = T - \tau + 1}^T f(x(t)) - f(x^\star).
  \nonumber
\end{align}

We may set $\tau = 0$ in the
expression~\eqref{eqn:expected-function-to-regret}, taking expectations and
applying Lemma~\ref{lemma:regret-bound}, to recover the known convergence
rates~\cite{NemirovskiJuLaSh09} for the stochastic gradient method with
independent samples. However, the essential idea that the
expansion~\eqref{eqn:expected-function-to-regret} allows us to implement is
that for large enough $\tau$, the sample $\statsample_{t + \tau}$ is nearly
``independent'' of the parameters $x(t)$, since the stochastic process
$\statprob$ is mixing. By allowing $\tau > 0$, we can bound the four
sums~\eqref{eqn:expected-function-to-regret} using a combination of
Lemmas~\ref{lemma:regret-bound} and~\ref{lemma:xt-diff}, then apply the mixing
properties of the stochastic process $\statprob$ to show that
$F(x(t); \statsample_{t + \tau})$ is a nearly unbiased estimate of $f(x(t))$:
\begin{equation*}
  \E[f(x(t)) - F(x(t); \statsample_{t + \tau})] \approx 0.
\end{equation*}
We formalize this intuition with two lemmas, whose proofs we
provide in
Appendix~\ref{app:proofs-expected-convergence}. 
\begin{lemma}
  \label{lemma:distance-expectations-close}
  Let $x$ be $\mc{F}_t$-measurable and $\tau \ge 1$.  If
  Assumption~\ref{assumption:one-step-variance} holds,
  \begin{equation*}
    \left|\E\left[f(x) - f(\opt) - F(x; \statsample_{t + \tau}) +
      F(\opt; \statsample_{t + \tau}) \mid \mc{F}_t\right] \right| \le
    3 \lipobj \radius \cdot \dhel\left(\statprob_{[t]}^{t + \tau},
    \stationary\right).
  \end{equation*}
  If Assumption~\ref{assumption:F-lipschitz} holds,
  \begin{equation*}
    \left|\E\left[f(x) - f(\opt) - F(x; \statsample_{t + \tau}) +
      F(\opt; \statsample_{t + \tau}) \mid \mc{F}_t\right] \right| \le
    \lipobj \radius \cdot \dtv\left(\statprob_{[t]}^{t + \tau},
    \stationary\right).
  \end{equation*}
\end{lemma}

The next lemma applies a type of stability argument, showing that
function values between $x(t)$ and $x(t + \tau)$ cannot be too far apart.
\begin{lemma}
  \label{lemma:single-step-variance-stability}
  Let $\tau \ge 0$ and $\stepsize(t)$ be non-increasing. If
  Assumption~\ref{assumption:one-step-variance} holds, then
  \begin{equation*}
    \E[F(x(t); \statsample_{t + \tau})
      - F(x(t + \tau); \statsample_{t + \tau}) \mid \mc{F}_{t-1}]
    \le \tau \stepsize(t) \lipobj^2.
  \end{equation*}
  If Assumption~\ref{assumption:F-lipschitz} holds, then
  \begin{equation*}
    F(x(t); \statsample_{t + \tau}) - F(x(t + \tau); \statsample_{t + \tau})
    \le \tau \stepsize(t) \lipobj^2.
  \end{equation*}
\end{lemma}

We can now apply
Lemmas~\ref{lemma:regret-bound}--\ref{lemma:single-step-variance-stability} to
give the promised proof of Theorem~\ref{theorem:expected-convergence}.

\begin{proof-of-theorem}[\ref{theorem:expected-convergence}]
  The equality~\eqref{eqn:expected-function-to-regret} is
  non-probabilistic, so all we need to complete the proof is to take
  expectations, applying the preceding lemmas. First, we map $\tau$ to
  $\tau - 1$ in the previous results, which will make our analysis
  cleaner. Throughout this proof, the quantity $d(\cdot, \stationary)$
  will denote $3 \dhel(\cdot, \stationary)$ when we apply
  Assumption~\ref{assumption:one-step-variance} and will denote
  $\dtv(\cdot, \stationary)$ when using
  Assumption~\ref{assumption:F-lipschitz}, as the proof is identical
  in either case.  We control the expectation of each of the four
  sums~\eqref{eqn:expected-function-to-regret} in turn.  First, we
  apply Lemma~\ref{lemma:distance-expectations-close} to see that
  \begin{equation*}
    \sum_{t = 1}^{T - \tau + 1} \E[f(x(t)) - f(\opt) - F(x(t);
      \statsample_{t + \tau - 1}) + F(\opt; \statsample_{t + \tau -
        1})] \le \lipobj \radius \sum_{t = 0}^{T - \tau + 1}
    \E[d(\statprob_{[t]}^{t + \tau}, \stationary)].
  \end{equation*}
  The second of the four sums~\eqref{eqn:expected-function-to-regret}
  requires Lemma~\ref{lemma:single-step-variance-stability}, which yields
  \begin{equation*}
    \sum_{t = 1}^{T - \tau + 1} \E[F(x(t); \statsample_{t + \tau - 1})
      - F(x(t + \tau - 1); \statsample_{t + \tau - 1})]
    \le (\tau - 1) \lipobj^2 \sum_{t = 1}^{T - \tau + 1} \stepsize(t).
  \end{equation*}
  Lemma~\ref{lemma:regret-bound} controls the third term in the
  series~\eqref{eqn:expected-function-to-regret}, and taking
  expectations gives $\E[\dnorm{g(t)}^2] \le \lipobj^2$. The final
  term in the sum~\eqref{eqn:expected-function-to-regret} is bounded
  by $(\tau - 1) \radius \lipobj$ when either of the Lipschitz
  assumptions~\ref{assumption:one-step-variance}
  or~\ref{assumption:F-lipschitz} hold.  Summing our four bounds, we
  obtain that for any $\tau \ge 1$,
  \begin{align}
    \E\bigg[\sum_{t = 1}^T f(x(t)) - f(x^\star)\bigg] & \le \lipobj
    \radius \sum_{t = 1}^{T - \tau + 1} \E\left[d\left(\statprob_{[t -
          1]}^{t + \tau - 1}, \stationary\right)\right] +
    \frac{\radius^2}{2 \stepsize(T)} + \frac{\lipobj^2}{2} \sum_{t =
      \tau}^T \stepsize(t) \nonumber \\ & \qquad ~ + (\tau - 1)
    \lipobj^2 \sum_{t = 1}^{T - \tau + 1} \stepsize(t) + (\tau - 1)
    \radius \lipobj.
    \label{eqn:almost-expected-theorem}
  \end{align}

  Assumption~\ref{assumption:uniform-mixing} states that
  there exists a uniform mixing time $\tmix(\statprob, \epsilon)$ (for both
  total variation and Hellinger mixing) such that $d(\statprob_{[t - 1]}^{t +
    \tau - 1}, \stationary) \le \epsilon$. Applying the
  definition of $\tmix$ for Hellinger or total variation mixing completes
  the proof.
\end{proof-of-theorem}

\subsection{High-probability convergence}
\label{sec:proof-of-theorem-highprob}

In this section, we complement the convergence bounds in
Section~\ref{sec:proof-of-theorem-expected} with high-probability
statements. We use martingale theory to show that the bound of
Theorem~\ref{theorem:expected-convergence} holds with high
probability.  We begin from the same starting point as the proof of
Theorem~\ref{theorem:expected-convergence}---with the
expansion~\eqref{eqn:expected-function-to-regret}---but now we show
that the random sum
\begin{equation}
  \label{eqn:highprob-sums-to-bound}
  \sum_{t = 1}^{T - \tau + 1} f(x(t)) - f(\opt) - F(x(t);
  \statsample_{t + \tau - 1}) + F(x^\star; \statsample_{t + \tau - 1})
\end{equation}
is small with high probability. Intuitively, this follows because given the
initial \mbox{$t - \tau$} samples $\statsample_1, \ldots, \statsample_{t - \tau}$,
the $t$th sample $\statsample_t$ is almost a sample from the stationary
distribution $\stationary$. With this in mind, we can show that an
appropriately subsampled version of the above sequence behaves
approximately as a martingale, and we can then apply Azuma's
inequality~\cite{Azuma67} to derive high-probability guarantees on
the sum~\eqref{eqn:highprob-sums-to-bound}.
\begin{proposition}
  \label{proposition:highprob-convergence}
  Let Assumption~\ref{assumption:F-lipschitz} hold and $\delta \in (0,
  1)$. With probability at least $1 - \delta$, for $\tau \in \N$ with
  $\tau \in [1, T/2]$,
  \begin{align*}
    \lefteqn{\sum_{t = 1}^{T - \tau + 1} \left[f(x(t)) - F(x(t);
        \statsample_{t + \tau - 1}) + F(x^\star; \statsample_{t + \tau
          - 1}) - f(x^\star)\right]} \\ & \qquad\qquad\qquad ~ \le 4
    \lipobj \radius \sqrt{T \tau \log\frac{\tau}{\delta}} + \lipobj
    \radius \sum_{t = 1}^T \dtv\left(\statprob_{[t - \tau]}^t,
    \stationary\right).
  \end{align*}
\end{proposition}

We provide a proof in
Appendix~\ref{app:highprob-convergence} and can now prove
Theorem~\ref{theorem:highprob-convergence}.

\begin{proof-of-theorem}[\ref{theorem:highprob-convergence}]
  The proof is a combination of the proofs of previous
  results. Starting from the expansion~\eqref{eqn:expected-function-to-regret},
  we use Lemma~\ref{lemma:single-step-variance-stability} to see that
  \begin{equation*}
    \sum_{t = 1}^{T - \tau + 1} F(x(t); \statsample_{t + \tau - 1})
    - F(x(t + \tau - 1); \statsample_{t + \tau - 1})
    \le (\tau - 1) \lipobj^2 \sum_{t = 1}^{T - \tau + 1} \stepsize(t),
  \end{equation*}
  and applying the $\lipobj$-Lipschitz continuity of the functions $F(\cdot;
  \statsample)$ and compactness of $\xdomain$ we obtain
  \begin{equation*}
    \sum_{t = T - \tau + 2}^T f(x(t)) - f(x^\star) \le (\tau - 1) \lipobj \radius.
  \end{equation*}
  In addition, the convergence guarantee in Lemma~\ref{lemma:regret-bound}
  guarantees that
  \begin{equation*}
    \sum_{t = \tau}^T F(x(t); \statsample_t) - F(x^\star; \statsample_t)
    \le \frac{1}{2 \stepsize(T)} \radius^2 +
    \frac{\lipobj^2}{2} \sum_{t = 1}^T \stepsize(t).
  \end{equation*}
  Combining these bounds, we can replace the
  equality~\eqref{eqn:expected-function-to-regret} with the bound
  \begin{align}
    \sum_{t = 1}^T f(x(t)) - f(x^\star)
    & \le \frac{1}{2 \stepsize(T)} \radius^2 +
    \frac{\lipobj^2}{2} \sum_{t = 1}^T \stepsize(t)
    + (\tau - 1)\bigg[\lipobj \radius
    + \lipobj^2 \sum_{t = 1}^{T} \stepsize(t)\bigg]
    \label{eqn:bound-to-concentrate} \\
    & \quad ~ + \sum_{t = 1}^{T - \tau + 1}
    \left[f(x(t)) - F(x(t); \statsample_{t + \tau - 1})
        + F(x^\star; \statsample_{t + \tau - 1})
        - f(x^\star)\right],
    \nonumber
  \end{align}
  which holds for any $\tau \ge 1$. What remains is to replace the last term
  in the non-probabilistic bound~\eqref{eqn:bound-to-concentrate} with the
  upper bound in Proposition~\ref{proposition:highprob-convergence}, which
  holds with probability $1 - \delta$, and then to replace $\tau$ with
  $\tmixtv(\statprob, \epsilon)$, which guarantees the inequality
  $\dtv(\statprob_{[t - \tau]}^t, \stationary) \le \epsilon$.
\end{proof-of-theorem}

\subsection{Random mixing}
\label{sec:proof-of-theorem-probabilistic-mixing}

In this section, we give the proof of
Theorem~\ref{theorem:probabilistic-mixing}. The proof
is similar to that of Theorem~\ref{theorem:highprob-convergence}, but
we need an auxiliary lemma that allows us to guarantee that
the mixing times are bounded uniformly for all times and for all desired
accuracies of mixing $\epsilon$. See Appendix~\ref{appendix:probabilistic-mixing}
for the proof of the lemma.
\begin{lemma}
  \label{lemma:uniform-probabilistic-mixing}
  Let Assumption~\ref{assumption:probabilistic-mixing} hold and
  $\delta \in (0, 1)$. With
  probability at least $1 - \delta$,
  \begin{equation*}
    \max_{s \in \{1, \ldots, T\}} \sup_{\{\epsilon \,:\,
      \tmixtv(\statprob, \epsilon) \le T\} } \left(
    \tmixtv(\statprob_{[s]}, \epsilon) - \tmixtv(\statprob,
    \epsilon)\right) \le \mixconst\left(\log\frac{1}{\delta} + 2
    \log(T)\right).
  \end{equation*}
\end{lemma}

Rewriting Lemma~\ref{lemma:uniform-probabilistic-mixing} slightly, we may
define $\tau =
\tmixtv(\statprob, \epsilon) + \mixconst(\log\frac{1}{\delta} + 2
\log(T))$, and we find that with probability at least $1 - \delta$, 
\begin{equation}
  \dtv\left(\statprob_{[s]}^{s + \tau}, \stationary\right)
  \le \epsilon
  \label{eqn:uniform-mixing-distances}
\end{equation}
for all $s \in \{1, \ldots, T\}$ and for all $\epsilon > 0$ with
$\tmixtv(\statprob, \epsilon) \le T$.
This leads us to

\begin{proof-of-theorem}[\ref{theorem:probabilistic-mixing}]
  All that is different in the proof of this theorem from that of
  Theorem~\ref{theorem:highprob-convergence} is that in the penultimate
  inequality~\eqref{eqn:bound-to-concentrate}, when we apply
  Proposition~\ref{proposition:highprob-convergence}, we no longer have the
  guarantee that $\dtv(\statprob_{[t - \tau]}^t, \stationary) \le
  \epsilon$ for all $t$. To that end, let $\epsilon$ be such that
  $\tmixtv(\statprob, \epsilon) \le T$. Apply
  Lemma~\ref{lemma:uniform-probabilistic-mixing} and its
  consequence~\eqref{eqn:uniform-mixing-distances}, which states that if we
  take $\tau = \tmixtv(\statprob, \epsilon) + \mixconst(\log\frac{1}{\delta} +
  2 \log(T))$, then we obtain that $\dtv(\statprob_{[t - \tau]}^t,
  \stationary) \le \epsilon$ with probability at least $1 - \delta$.  If
  $\tmixtv(\statprob, \epsilon) > T$, the bound in the theorem holds
  vacuously, so we may extend the result to all $\epsilon > 0$.
\end{proof-of-theorem}

\subsection{Lower bounds on optimization accuracy}
\label{sec:lower-bound-proof}

Our proof of Theorem~\ref{theorem:lower-bound} mirrors the proof of Theorem~1
in the paper by Agarwal et al.~\cite{AgarwalBaRaWa12}, so we are somewhat
terse in our description and proof. The intuition in the proof is that if the
stochastic process $\statprob$ returns a sample from the stationary
distribution $\stationary$ every $\tau$ timesteps, otherwise returning a
sample identical to the previous one, then the convergence rate of any
algorithm should be a factor of $\tau$ slower than if it could receive
independent samples from $\stationary$. Mesterharm~\cite{Mesterharm05} employs
a similar approach to give a lower bound on the performance of online learning
algorithms. More formally, by using an identical construction to~\cite[Section
  IV.A]{AgarwalBaRaWa12}, we may reduce the problem of minimization of a
function $f : \R^d \rightarrow \R$ to identifying the bias of $d$ coins.  To
that end, let $\hypercubeset \subset \{-1, 1\}^d$ be a packing of the
$d$-dimensional hypercube such that $\cubecorner, \cubecorner' \in
\hypercubeset$ with $\cubecorner \neq \cubecorner'$ satisfy $\norm{\cubecorner
  - \cubecorner'}_1 \ge d/2$; it is a classical fact~\cite{Matousek02} that
there is such a set with cardinality $|\hypercubeset| \ge (2
e^{-\half})^{d/2}$.

Now for a fixed $\tau \in \N$, consider the following sequential sampling
procedure, which generates a set of pairs of random vectors $\{(U_t,
Y_t)\}_{t = 1}^\infty$. Choose a vector \mbox{$\cubecorner \in \hypercubeset$}
uniformly at random and let $\delta \in
\openleft{0}{1/4}$. Let $\statprob_\cubecorner$ denote the distribution
(conditional on $\cubecorner$) that corresponds to the following: for each
$t$, construct samples according to
\begin{enumerate}[(a)]
\item \label{item:no-sample}
 If $(t - 1) \mod \tau \neq 0$, take $U_t = U_{t-1}$ and $Y_t = Y_{t-1}$.
\item \label{item:do-sample}
  Otherwise, pick a uniformly random subset $U_t \subset \{1, \ldots, d\}$
  of size $|U_t| = \numrevealed$, then
  \begin{enumerate}[(i)]
  \item For each $i \in U_t$, construct a random variable $C_i$ such
    that $C_i = 1$ with probability $\half +
    \cubecorner_i \delta$ and $C_i = -1$ with probability $\half -
    \cubecorner_i \delta$.
  \item Construct the vector $Y_t \in \{-1, 1\}^d$ such
    that $Y_{t,i} = C_i$ if $i \in U_t$, and otherwise
    $Y_{t,i}$ is uniform Bernoulli, that is, if
    $i \not \in U_t$ then $Y_{t,i} = 1$ with
    probability $\half$ and $Y_{t,i} = -1$ with probability $\half$.
  \end{enumerate}
\end{enumerate}
This sampling procedure yields a sequence $\statsample_t = (U_t, Y_t)$,
where if $\stationary_\cubecorner$ is the distribution of a pair $(U, Y)$ such
that $U \subset \{1, \ldots, d\}$ is chosen uniformly at random with size $|U|
= \numrevealed$ and $Y$ is sampled according to the steps (i)--(ii) above,
then $\stationary_\cubecorner$ is the stationary distribution of
$\statprob_\cubecorner$. Moreover, we see that $\dtv(\statprob_{[t]}^{t + \tau
  + k}, \stationary) = 0$ for any $k \ge 0$ and any $t$, since the
distribution $\statprob_\cubecorner$ corresponds to receiving an independent
sample $(U, Y)$ from $\stationary_\cubecorner$ every $\tau$ steps.

Let $\information(X; Y)$ denote the mutual information between random
variables $X$ and $Y$ and let $\entropy(X)$ denote the (Shannon) entropy of
$X$.  By inspection of Agarwal et al.'s proof~\cite[Lemma 3]{AgarwalBaRaWa12},
since $\stationary_\cubecorner$ is the stationary distribution of
$\statprob_\cubecorner$, a tight enough bound on the mutual information
$\information((U_1, Y_1), \ldots, (U_T, Y_T); \cubecorner)$ proves
Theorem~\ref{theorem:lower-bound}. Hence, we provide the following lemma:
\begin{lemma}
  \label{lemma:block-mutual-information}
  Let the sequence $\statsample_t = (U_t, Y_t)$ be generated
  according to the steps~\eqref{item:no-sample}--\eqref{item:do-sample} above.
  Then for $\delta \in \openleft{0}{1/4}$,
  \begin{equation*}
    \information\left((U_1, Y_1), \ldots, (U_T, Y_T);
    \cubecorner\right)
    \le 16 \ceil{\frac{T}{\tau}} \numrevealed \delta^2.
  \end{equation*}
\end{lemma}
\begin{proof}
  Our sampling model~\eqref{item:no-sample}--\eqref{item:do-sample} sets
  blocks of size $\tau$ to be equal, that is, $(U_1, Y_1) = \cdots = (U_\tau,
  Y_\tau)$, $(U_{\tau + 1}, Y_{\tau + 1}) = \cdots = (U_{2\tau}, Y_{2\tau})$,
  and so on, whereas different blocks are independent given the variable
  $\cubecorner$. We thus see that by the definitions of mutual information,
  conditional entropy, and that
  entropy is sub-additive~\cite{CoverTh91},
  \begin{align}
    \lefteqn{\information\left((U_1, Y_1), \ldots, (U_T, Y_T);
      \cubecorner\right)} \nonumber \\
    & = \entropy((U_1, Y_1), \ldots, (U_T, Y_T))
    - \entropy((U_1, Y_1), \ldots, (U_T, Y_T)
    \mid \cubecorner) \nonumber \\
    & = \entropy((U_1, Y_1), \ldots, (U_T, Y_T))
    - \sum_{k = 1}^{\ceil{T/\tau}}
    \entropy\left((U_{(k - 1)\tau + 1}, Y_{(k - 1)\tau + 1}),
    \ldots, (U_{k\tau}, Y_{k\tau})
    \mid \cubecorner\right) \nonumber \\
    \ifdefined\siam %
    & \le \sum_{k=1}^{\ceil{T/\tau}} \Big[
      \entropy\left(((U_{(k - 1)\tau + 1}, Y_{(k - 1)\tau + 1}),
    \ldots, (U_{k\tau}, Y_{k\tau})\right) \nonumber \\
    & \qquad\qquad~ - \entropy\left((U_{(k - 1)\tau + 1}, Y_{(k - 1)\tau + 1}),
    \ldots, (U_{k\tau}, Y_{k\tau}) \mid \cubecorner\right)
    \Big] \nonumber \\
    \else %
    & \le \sum_{k = 1}^{\ceil{T/\tau}}
    \left[\entropy\left((U_{(k - 1)\tau + 1}, Y_{(k - 1)\tau + 1}),
    \ldots, (U_{k\tau}, Y_{k\tau})\right)
    - \entropy\left((U_{(k - 1)\tau + 1}, Y_{(k - 1)\tau + 1}),
    \ldots, (U_{k\tau}, Y_{k\tau}) \mid \cubecorner\right)
    \right] \nonumber \\
    \fi
    & = \sum_{k = 1}^{\ceil{T/\tau}} \information\left(
    (U_{(k - 1)\tau + 1}, Y_{(k - 1)\tau + 1}),
    \ldots, (U_{k\tau}, Y_{k\tau}); \cubecorner \right)
    = \sum_{k = 1}^{\ceil{T/\tau}} \information\left(
    (U_{k\tau}, Y_{k\tau}); \cubecorner \right).
    \label{eqn:block-information-bound}
  \end{align}
  In the last line we have used that within the same block of size $\tau$, all
  $(U_t, Y_t)$ pairs are equal.  Now, using the
  bound~\eqref{eqn:block-information-bound}, we apply an identical derivation
  as that given in the proof of Agarwal et al.'s Lemma 3 (following Eq.~(25)
  there). For any fixed $k$ we have $\information((U_{k\tau}, Y_{k\tau});
  \cubecorner) \le 16 \numrevealed \delta^2$, which completes the proof of the
  lemma.
\end{proof}

\begin{proof-of-theorem}[\ref{theorem:lower-bound}]
  Use Agarwal et al.'s construction (see Eq.~(16) in Section IV.A
  of~\cite{AgarwalBaRaWa12}) of a ``difficult'' subclass of functions,
  then in the proof of Theorem~1 from~\cite{AgarwalBaRaWa12}, replace their
  coin-flipping oracle with the
  steps~\eqref{item:no-sample}--\eqref{item:do-sample} and applications of
  their Lemma~3 with Lemma~\ref{lemma:block-mutual-information} above.
\end{proof-of-theorem}

%% file: conclusions.tex
\section{Conclusions}

In this paper, we have shown that stochastic subgradient and mirror
descent approaches extend in an elegant way to situations in which we
have no access to i.i.d.\ samples from the desired distribution. In
spite of this difficulty, we are able to achieve reasonably fast rates
of convergence for the ergodic mirror descent algorithm---the natural
extension of stochastic mirror descent---under reasonable assumptions
on the ergodicity of the stochastic process $\{\statsample_t\}$ that
generates the samples. We gave several examples showing the strengths
and uses of our new analysis, and believe that there are many more.
\ifdefined\siam
In the extended version of this paper~\cite{DuchiAgJoJo11}, we give additional
applications to optimization over combinatorial spaces where Markov chain
Monte Carlo samplers can be designed to sample efficiently from the
space~\cite{JerrumSi96}.
\fi
In addition, our results give a relatively clean and simple
way to derive finite sample rates of convergence for statistical
estimators with dependent data without requiring the full machinery of
empirical process theory~(e.g.,~\cite{Yu94}).  Though we have provided
lower bounds showing that our analysis is tight to numerical
constants, it may be possible to sharpen our results for interesting
special cases, such as when the distribution of the stochastic process
$\{\statsample_t\}$ has nice enough Markovianity properties.  We leave
such questions to future work.

%% file: appendix.tex
\section{Proofs of Optimization Results}
\label{appendix:md-proofs}

\begin{proof-of-lemma}[\ref{lemma:regret-bound}]
  The proof of the lemma begins by controlling the amount of progress made by
  one step of the EMD method, then summing the resulting bound.
  By the first-order convexity inequality and definition of the
  subgradient $g(t)$, we have
  \begin{align}
    F(x(t); \statsample_t) - F(x^*; \statsample_t)
    & \le \<g(t), x(t) - x^*\> \nonumber \\
    & = \<g(t), x(t + 1) - x^*\> + \<g(t), x(t + 1) - x(t)\>.
    \label{eqn:expand-gradient-md}
  \end{align}
  For $y \in \xdomain$, the first-order optimality conditions
  for $x(t + 1)$ in the update~\eqref{eqn:ergodic-md}
  imply
  \begin{equation*}
    \<\stepsize(t) g(t) + \nabla\prox(x(t + 1)) - \nabla\prox(x(t)),
    y - x(t + 1)\> \ge 0.
  \end{equation*}
  In particular, we can take $y = x^*$ in this
  bound to find
  \begin{equation}
    \label{eqn:single-step-md-optimality}
    \stepsize(t) \<g(t), x(t + 1) - x^*\> \le \< \nabla \prox(x(t + 1)) -
    \nabla\prox(x(t)), x^* - x(t + 1)\>.
  \end{equation}
  Now we use the definition of the Bregman divergence $\divergence$, 
  to obtain 
  \begin{align*}
    \lefteqn{\<\nabla \prox(x(t + 1)) - \nabla \prox(x(t)), x^* - x(t + 1)\>}
    \\ & \qquad\qquad\qquad\qquad ~
    = \divergence(x^*, x(t)) - \divergence(x^*, x(t + 1))
    - \divergence(x(t + 1), x(t)).
  \end{align*}
  Combining this result with the expanded gradient
  term~\eqref{eqn:expand-gradient-md} and the the first-order convexity
  inequality~\eqref{eqn:single-step-md-optimality}, we get
  \begin{align*}
    \lefteqn{F(x(t); \statsample_t) - F(x^*; \statsample_t)
    \le \frac{1}{\stepsize(t)} \divergence(x^*, x(t))
    - \frac{1}{\stepsize(t)} \divergence(x^*, x(t + 1))}
    \\
    & \qquad\qquad\qquad\qquad\qquad\quad
    ~ - \frac{1}{\stepsize(t)} \divergence(x(t + 1), x(t))
    + \<g(t), x(t + 1) - x(t)\> \\
    & \qquad\qquad
    \stackrel{(i)}{\le} \frac{1}{\stepsize(t)} \divergence(x^*, x(t))
    - \frac{1}{\stepsize(t)} \divergence(x^*, x(t + 1))
    - \frac{1}{\stepsize(t)} \divergence(x(t + 1), x(t)) \\
    & \qquad\qquad\quad ~
    + \frac{\stepsize(t)}{2} \dnorm{g(t)}^2
    + \frac{1}{2\stepsize(t)} \norm{x(t + 1) - x(t)}^2 \\
    & \qquad\qquad 
    \stackrel{(ii)}{\le} \frac{1}{\stepsize(t)} \divergence(x^*, x(t))
    - \frac{1}{\stepsize(t)} \divergence(x^*, x(t + 1))
    + \frac{\stepsize(t)}{2} \dnorm{g(t)}^2.
  \end{align*}
  The inequality $(i)$ is a consequence of the Fenchel-Young inequality
  applied to the conjugates $\half \norm{\cdot}^2$ and $\half \dnorm{\cdot}^2$
  (see, e.g.,~\cite[Example 3.27]{BoydVa04}), while the inequality $(ii)$
  follows by the strong convexity of $\prox$,
  which gives $\divergence(x(t + 1), x(t)) \ge \half \norm{x(t + 1) -
    x(t)}^2$.

  Summing the final inequality, we obtain
  \begin{align*}
    \lefteqn{\sum_{t = \tau + 1}^T \left[F(x(t); \statsample_t)
      - F(x^*; \statsample_t)\right]} \\
    & \qquad \le \sum_{t = \tau + 1}^T
    \frac{1}{\stepsize(t)}
    \left[\divergence(x^*, x(t)) - \divergence(x^*, x(t + 1))\right]
    + \sum_{t = \tau + 1}^T \frac{\stepsize(t)}{2} \dnorm{g(t)}^2.
  \end{align*}
  Using the compactness assumption that $\divergence(x^*, x) \le \half
  \radius^2$ for all $x \in \xdomain$, we have
  \begin{align*}
    \lefteqn{\sum_{t = \tau + 1}^T \frac{1}{\stepsize(t)}
      \left[\divergence(x^*, x(t)) - \divergence(x^*, x(t + 1))\right]}
    \\
    & \le \sum_{t = \tau + 2}^T \divergence(x^*, x(t))
    \left[\frac{1}{\stepsize(t)} - \frac{1}{\stepsize(t - 1)}\right]
    + \frac{1}{\stepsize(\tau + 1)} \divergence(x^*, x(\tau + 1)) \\
    & \le \frac{\radius^2}{2} \sum_{t = \tau + 2}^T
    \left[\frac{1}{\stepsize(t)} - \frac{1}{\stepsize(t - 1)}\right]
    + \frac{1}{2\stepsize(\tau + 1)} \radius^2
    = \frac{\radius^2}{2 \stepsize(T)},
  \end{align*}
  where for the last inequality we used that the stepsizes $\stepsize(t)$
  are non-increasing.
\end{proof-of-lemma}

\begin{proof-of-lemma}[\ref{lemma:xt-diff}]
  By the first-order condition for the optimality of $x(t + 1)$ for the
  update~\eqref{eqn:ergodic-md}, we have
  \begin{equation*}
    \<\stepsize(t) g(t) + \nabla \prox(x(t + 1)) -
    \nabla\prox(x(t)), x(t) - x(t + 1)\> \ge 0.
  \end{equation*}
  Rewriting, we have
  \begin{align*}
    \<\nabla \prox(x(t)) - \nabla\prox(x(t + 1)), x(t) - x(t + 1)\>
    & \le \stepsize(t) \<g(t), x(t) - x(t + 1)\> \\
    & \le \stepsize(t) \dnorm{g(t)} \norm{x(t) - x(t + 1)}
  \end{align*}
  using H\"older's inequality. Simple algebra shows that
  \begin{equation*}
    \divergence(x(t), x(t + 1)) + \divergence(x(t + 1), x(t))
    = \<\nabla \prox(x(t)) - \nabla \prox(x(t + 1)), x(t) - x(t + 1)\>,
  \end{equation*}
  and by the assumed strong convexity of $\prox$, we see
  \begin{align*}
    \norm{x(t) - x(t + 1)}^2
    & \le \divergence(x(t + 1), x(t)) + \divergence(x(t), x(t + 1)) \\
    & \le \stepsize(t) \dnorm{g(t)} \norm{x(t) - x(t + 1)}.
  \end{align*}
  Dividing by $\norm{x(t) - x(t + 1)}$ gives the desired result.
\end{proof-of-lemma}

\section{Mixing and expected function values}
\label{app:proofs-expected-convergence} 

\begin{proof-of-lemma}[\ref{lemma:distance-expectations-close}]
  Since $x \in \mc{F}_t$, we may integrate only against $\statsample$ when
  taking expectations, which yields
  \begin{align*}
    \lefteqn{\E\left[f(x) - f(\opt) - F(x; \statsample_{t + \tau}) + F(\opt;
      \statsample_{t + \tau}) \mid \mc{F}_t\right]} \\
    & \quad = \int (F(x;
    \statsample) - F(\opt; \statsample)) d\stationary(\statsample)
    - \int (F(x; \statsample) - F(\opt; \statsample))
    d\statprob^{t + \tau}_{[t]}(\statsample).
  \end{align*}
  Since we assume $\statprob_{[s]}^t$ and $\stationary$ have densities
  $\statdensity_{[s]}^t$ and $\stationarydensity$ with respect to a
  measure $\mu$, this difference becomes \mbox{$\int (F(x;
    \statsample) - F(\opt; \statsample))
    (\stationarydensity(\statsample) - \statdensity_{[t]}^{t +
      \tau}(\statsample)) d\mu(\statsample)$}. Setting $\statdensity =
  \statdensity_{[t]}^{t + \tau}$ for shorthand, we obtain
  \begin{align*}
    \lefteqn{\left|\int (F(x; \statsample) - F(\opt; \statsample))
      (\stationarydensity(\statsample) - \statdensity(\statsample))
      d\mu(\statsample)\right| \le \int |F(x; \statsample) - F(\opt;
      \statsample)| \left|\statdensity(\statsample) -
      \stationarydensity(\statsample)\right|d\mu(\statsample)}
    \\ 
    & = \int |F(x; \statsample) - F(\opt; \statsample)|
    \left(\sqrt{\stationarydensity(\statsample)} +
    \sqrt{\statdensity(\statsample)}\right)
    \left|\sqrt{\stationarydensity(\statsample)} -
    \sqrt{\statdensity(\statsample)}\right| d\mu(\statsample)
    \\
    & \le \sqrt{\int \left( F(x; \statsample) - F(\opt; \statsample)  
    \right)^2 \left( \sqrt{\stationarydensity(\statsample)} +
    \sqrt{\statdensity(\statsample)}\right)^2
    d\mu(\statsample) \int \left(\sqrt{\statdensity(\statsample)} - 
    \sqrt{\stationarydensity(\statsample)}\right)^2 d\mu(\statsample)} \\ 
    & = \sqrt{\int (F(x; \statsample) - F(\opt; \statsample))^2
      \left(\sqrt{\stationarydensity(\statsample)}
      + \sqrt{\statdensity(\statsample)}\right)^2
      d \mu(\statsample)}
    \,\dhel(\statprob_{[t]}^{t + \tau}, \stationary)
  \end{align*}
  by H\"older's inequality. Applying the inequality $(a + b)^2 \le 2a^2 +
  2b^2$, valid for $a, b \in \R$, we obtain the further bound
  \begin{align}
    \lefteqn{\left(2 \int \left(F(x; \statsample) - F(\opt; \statsample)
      \right)^2(\stationarydensity(\statsample) +
      \statdensity(\statsample)) d\mu(\statsample)\right)^{\half}
      \dhel(\statprob_{[t]}^{t + \tau}, \stationary) = }
    \label{eqn:bound-conditional-hellinger} \\
    & \sqrt{2} \left(\E_\stationary[\left( F(x; \statsample) -
      F(\opt; \statsample) \right)^2] + \E[\left( F(x; \statsample_{t
        + \tau}) - F(\opt; \statsample_{t + \tau}) \right)^2 \mid
      \mc{F}_t]\right)^{\half} \dhel(\statprob_{[t]}^{t + \tau},
    \stationary).
    \nonumber
  \end{align}

  To control the expectation terms in the
  bound~\eqref{eqn:bound-conditional-hellinger}, we now use
  Assumption~\ref{assumption:one-step-variance}. By the ($\statprob$-almost
  sure) convexity of the function $x \mapsto F(x;\statsample)$, we observe
  that
  \begin{equation*}
    F(x;\statsample) - F(\opt; \statsample) \leq
    \<\gradfunc(x;\statsample), x - \opt \>
    ~~~ \mbox{and} ~~~
    F(\opt; \statsample) - F(x;\statsample) \leq \<\gradfunc(\opt;
    \statsample), \opt - x\>. 
  \end{equation*}
  Combining these two inequalities, we see that 
  \begin{align*}
    \left( F(x;\statsample) - F(\opt; \statsample) \right)^2 &\leq \max
    \left\{ \<\gradfunc(x;\statsample), x - \opt \>^2,
    \<\gradfunc(\opt; \statsample), \opt - x\>^2 \right\}\\
    &\leq \max \left\{ \dnorm{\gradfunc(x;\statsample)}^2, 
    \dnorm{\gradfunc(\opt; \statsample)}^2 \right\} \norm{x -
      \opt}^2\\
    &\leq \max \left\{ \dnorm{\gradfunc(x;\statsample)}^2, 
    \dnorm{\gradfunc(\opt; \statsample)}^2 \right\} \radius^2,
  \end{align*}
  where the last inequality uses our compactness
  assumption~(\ref{eqn:compactness}). Now we invoke
  Assumption~\ref{assumption:one-step-variance} combined with the
  above inequality to obtain the further bound
  \begin{align*}
    \lefteqn{
      \E\left[\left( F(x; \statsample_{t + \tau}) - F(\opt; \statsample_{t + \tau})
        \right)^2 \mid \mc{F}_t\right]} \\
    & \qquad
    \le \radius^2 \E\left[\dnorm{\gradfunc(x;\statsample_{t + \tau})}^2 +
      \dnorm{\gradfunc(\opt; \statsample_{t + \tau})}^2\mid \mc{F}_t\right]
    \le
    2\lipobj^2\radius^2.  
  \end{align*}
  An analogous argument yields the same bound for the expectation
  under the stationary distribution, so based on our
  earlier bound~\eqref{eqn:bound-conditional-hellinger} we have
  \begin{align*}
    \lefteqn{
      \left|
      \int (F(x; \statsample) - F(\opt; \statsample))
      \left(d\stationary(\statsample) - 
      d\statprob_{[t]}^{t + \tau}(\statsample)\right)\right|} \\
    & \le \int|F(x; \statsample) -
    F(\opt; \statsample)| \left|d\stationary(\statsample) - 
    d\statprob_{[t]}^{t + \tau}(\statsample)\right|
    \le \sqrt{8 \lipobj^2 \radius^2} \, \dhel\left(\statprob_{[t]}^{t + \tau},
    \stationary\right).
  \end{align*}
  This completes the proof of the first statement of the lemma.
  
  The second statement is simpler: apply
  Assumption~\ref{assumption:F-lipschitz} to obtain
  \begin{equation*}
    \left|\int (F(x; \statsample) - F(\opt; \statsample))
    (\stationarydensity(\statsample) - \statdensity(\statsample))
    d\mu(\statsample)\right| \le \lipobj \radius
    \int|\statdensity(\statsample) - \stationarydensity(\statsample)|
    d\mu(\statsample).
  \end{equation*}
  Observing that the above bound is equal to $\lipobj\radius
  \dtv\left(\statprob_{[t]}^{t + \tau}, \stationary\right)$ completes
  the proof.
\end{proof-of-lemma}\\

\begin{proof-of-lemma}[\ref{lemma:single-step-variance-stability}] 
  For any $x$ measurable with respect to the $\sigma$-field $\mc{F}_s$, we can
  define the function $h_{[s]}(x) = \E[F(x; \statsample_{s + 1}) \mid
    \mc{F}_s]$. Assumption~\ref{assumption:one-step-variance} implies
  that $h_{[s]}$ is a $\lipobj$-Lipschitz continuous function so long as its
  argument is $\mc{F}_s$-measurable, that is, $|h_{[s]}(x) - h_{[s]}(y)| \le
  \lipobj \norm{x - y}$ for $x, y \in \mc{F}_s$. In turn, this implies that
  \begin{align*}
    \lefteqn{\E[F(x(t); \statsample_{t + \tau}) -
        F(x(t + \tau); \statsample_{t + \tau}) \mid \mc{F}_{t-1}]} \\
    & \qquad = \sum_{s = t}^{t + \tau - 1}
    \E[F(x(s); \statsample_{t + \tau}) - F(x(s + 1); \statsample_{t + \tau})
      \mid \mc{F}_{t-1}] \\
    & \qquad
    = \sum_{s = t}^{t + \tau - 1}
    \E\left[\E[F(x(s); \statsample_{t + \tau}) - F(x(s + 1);
        \statsample_{t + \tau}) \mid \mc{F}_{t + \tau - 1}]
      \mid \mc{F}_{t - 1}\right] \\
    & \qquad
    \le \sum_{s = t}^{t + \tau - 1}
    \E\left[\lipobj \norm{x(s) - x(s + 1)} \mid \mc{F}_{t-1}\right]
  \end{align*}
  since $x(s)$ is $\mc{F}_{t + \tau - 1}$-measurable for $s \le t + \tau$.
  Now we apply Lemma~\ref{lemma:xt-diff}, which shows that $\norm{x(s) - x(s +
    1)} \le \stepsize(s)\dnorm{g(s)}$, and we
  have the further inequality
  \begin{equation*}
    \E[F(x(t); \statsample_{t + \tau})
      - F(x(t + \tau); \statsample_{t + \tau}) \mid \mc{F}_{t-1}]
    \le \sum_{s = t}^{t + \tau - 1} \lipobj \stepsize(s) \E[\dnorm{g(s)}
    \mid \mc{F}_{t-1}].
  \end{equation*}
  Applying Jensen's inequality and
  Assumption~\ref{assumption:one-step-variance}, we see that
  \begin{equation*}
    \E[\dnorm{g(s)} \mid \mc{F}_{t-1}]
    \le \sqrt{\E[\E[\dnorm{g(s)}^2 \mid \mc{F}_{s - 1}] \mid \mc{F}_{t-1}]}
    \le \sqrt{\lipobj^2} = \lipobj.
  \end{equation*}
  In conclusion, we have the first statement of the lemma:
  \begin{equation*}
    \E[F(x(t); \statsample_{t + \tau})
      - F(x(t + \tau); \statsample_{t + \tau}) \mid \mc{F}_{t-1}]
    \le \lipobj^2 \sum_{s = t}^{t + \tau - 1} \stepsize(s)
    \le \lipobj^2 \tau \stepsize(t),
  \end{equation*}
  since the sequence $\stepsize(t)$ is non-increasing. The proof of the
  second statement is entirely similar, but we do not need to
  apply conditional expectations.
\end{proof-of-lemma}

\section{Proof of Proposition~\ref{proposition:highprob-convergence}}
\label{app:highprob-convergence}

\begin{proof-of-proposition}[\ref{proposition:highprob-convergence}]
  We construct a family of $\tau$ different martingales from the summation in
  the statement of the proposition, each of which we control with high
  probability. Applying a union bound gives us control on the deviation of the
  entire series. We begin by defining the random variables
  \begin{equation*}
    Z_t \defeq f(x(t - \tau + 1)) - F(x(t - \tau + 1); \statsample_t)
    + F(x^\star; \statsample_t) - f(x^\star),
  \end{equation*}
  noting that
  \begin{equation*}
    \sum_{t = \tau}^T Z_t
    = \sum_{t = 1}^{T - \tau + 1}
    \left[f(x(t)) - F(x(t); \statsample_{t + \tau - 1})
    + F(x^\star; \statsample_{t + \tau - 1})
    - f(x^\star)\right].
  \end{equation*}
  By defining the filtration of $\sigma$-fields $\mc{A}_i^j = \mc{F}_{\tau i +
    j}$ for $j = 1, \ldots, \tau$, we can construct a set of Doob martingales
  $\{X_1^j, X_2^j, \ldots\}$ for $j = 1, \ldots, \tau$ by making the
  definition
  \begin{align*}
    X_i^j & \defeq
    Z_{\tau i + j} - \E[Z_{\tau i + j} \mid \mc{A}_{i-1}^j]
    = Z_{\tau i + j} - \E[Z_{\tau i + j} \mid \mc{F}_{\tau(i - 1) + j}] \\
    & ~ = f(x(\tau(i - 1) + j + 1)) - F(x(\tau(i - 1) + j + 1);
    \statsample_{\tau i + j}) \\
    & \qquad\quad ~ + F(x^\star; \statsample_{\tau i + j}) - f(x^\star)
    - \E[Z_t \mid \mc{F}_{\tau(i - 1) + j}].
  \end{align*}
  By inspection, $X_i^j$ is measurable with respect to the 
  $\sigma$-field $\mc{A}_i^j$, and
  $\E[X_i^j \mid \mc{A}_{i-1}^j] = 0$. So, for each $j$, the sequence
  $\{X_i^j : i = 1, 2, \ldots\}$ is a martingale difference sequence
  adapted to the
  filtration $\{\mc{A}_i^j : i = 1, 2, \ldots\}$.  Define the index set
  $\indset(j)$ to be the indices $\{1, \ldots, \floor{T/\tau} + 1\}$ for $j
  \le T - \tau\floor{T/\tau}$ and $\{1, \ldots, \floor{T/\tau}\}$ otherwise.
  With the definition of $X_i^j$ and the indices $\indset(j)$, we see that
  \begin{equation}
    \label{eqn:z-martingale-sequence}
    \sum_{t = \tau}^T Z_t
    = \sum_{j=1}^\tau \sum_{i \in \indset(j)} X_i^j
    + \sum_{t = \tau}^T \E[Z_t \mid \mc{F}_{t - \tau}]
    = \sum_{j=1}^\tau \sum_{i = 1}^{|\indset(j)|} X_i^j
    + \sum_{t = \tau}^T \E[Z_t \mid \mc{F}_{t - \tau}].
  \end{equation}

  Now we note the following important fact: by the compactness
  assumption~\eqref{eqn:consequence-compactness} and
  Assumption~\ref{assumption:F-lipschitz}, the $\mc{F}_{\tau(i - 1) +
    j}$-measurability of $f(x(\tau(i - 1) + j + 1))$ implies
  \begin{equation*}
    |X_i^j|
    = \left|Z_{\tau i + j} - \E[Z_{\tau i + j} \mid \mc{F}_{\tau(i - 1) + j}]\right|
    \le 2 \lipobj \radius.
  \end{equation*}
  This bound, coupled with the
  representation~\eqref{eqn:z-martingale-sequence}, shows that
  $\sum_{t=\tau}^T Z_t$ is a sum of $\tau$ different
  bounded-difference martingales plus a sum of conditional expectations 
  that we will bound later. To control the martingale portion of the
  sum~\eqref{eqn:z-martingale-sequence}, we apply the triangle inequality, a
  union bound, and Azuma's inequality~\cite{Azuma67} to find
  \begin{align*}
    \statprob\bigg(\sum_{j = 1}^\tau \sum_{i \in \indset(j)} X_i^j > \gamma\bigg)
    & \le \sum_{j=1}^\tau \statprob\bigg(\sum_{i \in \indset(j)} X_i^j >
    \frac{\gamma}{\tau}\bigg)
    \le \sum_{j=1}^\tau
    \exp\left(-\frac{\gamma^2}{16 \lipobj^2 \radius^2 \tau T}\right),
  \end{align*}
  since there are fewer than $2 T/\tau$ terms in each of the sums $X_i^j$
  (by our assumption that $T/2 \ge \tau$).
  Substituting $\gamma = 4\lipobj \radius
  \sqrt{T \tau \log (\tau/\delta)}$, we find
  \begin{equation*}
    \statprob\bigg(\sum_{j=1}^\tau \sum_{i \in \indset(j)}
    X_i^j > 4 \lipobj \radius \sqrt{T
      \tau \log \frac{\tau}{\delta}}
    \bigg)
    \le \delta.
  \end{equation*}
  To bound the final term $\E[Z_t \mid \mc{F}_{t - \tau}]$ in the
  sum~\eqref{eqn:z-martingale-sequence}, we recall from 
  Lemma~\ref{lemma:distance-expectations-close} that
  \begin{equation*}
    \left|\E[Z_t \mid \mc{F}_{t - \tau}]\right| \le \lipobj \radius
    \cdot \dtv\left(\statprob_{[t - \tau]}^t, \stationary \right).
  \end{equation*}
  Summing this bound completes the proof.
\end{proof-of-proposition}

\section{Probabilistic Mixing}
\label{appendix:probabilistic-mixing}

\begin{proof-of-lemma}[\ref{lemma:beta-mixing-to-probabilities}]
  Using the definitions in the statement of the lemma, take
  \begin{equation*}
    \tau = \floor{\tmixtv(\statprob, \epsilon) + \mixconst c}
    \ge \frac{\log \frac{1}{\epsilon}}{|\log \gamma|}
    + \frac{\log K}{|\log \gamma|} + \frac{c}{|\log \gamma|},
  \end{equation*}
  which implies by Markov's inequality that
  \begin{equation*}
    \P\left(\dtv\left(\statprob_{[t]}^{t + \tau}, \stationary\right)
    \ge \epsilon\right) \le \frac{K \gamma^\tau}{\epsilon}
    \le \frac{K \exp(-\log\frac{1}{\epsilon})
      \exp(-\log K)}{\epsilon} \exp(-c)
    = \exp(-c)
  \end{equation*}
  since $\gamma^{a / |\log \gamma|} = \exp(-a)$ for $0 < \gamma < 1$.
  Noting that
  \begin{equation*}
    \P\left(\tmixtv(\statprob_{[t]}, \epsilon) > \tau\right)
    \le
    \P\left(\dtv\left(\statprob_{[t]}^{t + \tau}, \stationary\right)
    > \epsilon\right)
  \end{equation*}
  for any $\tau \in \N$ completes the proof.
\end{proof-of-lemma}\\

\begin{proof-of-lemma}[\ref{lemma:uniform-probabilistic-mixing}]
  We use a covering number argument, which is common in
  uniform concentration inequalities in probability theory
  (e.g.,~\cite{VapnikCh71}). For each $t \in \{1, \ldots, T\}$, define
  \begin{equation*}
    \epsilon_t \defeq \inf\left\{\epsilon > 0 : \tmixtv(\statprob, \epsilon)
    \le t\right\}.
  \end{equation*}
  By the right-continuity of $\epsilon \mapsto \tmixtv(\statprob, \epsilon)$,
  we have $\tmixtv(\statprob, \epsilon_t) \le t$ but $\tmixtv(\statprob,
  \epsilon_t - \delta) > t$ for any $\delta > 0$.  As a consequence, we see
  that for some $\epsilon \ge \epsilon_T$ to exist satisfying
  $\tmixtv(\statprob_{[s]}, \epsilon) > \tmixtv(\statprob, \epsilon) + c$, it
  must be the case that
  \begin{equation*}
    \tmixtv(\statprob_{[s]}, \epsilon_t) - \tmixtv(\statprob, \epsilon_t)
    > c
  \end{equation*}
  for some $\epsilon_t$, where $t \in \{1, \ldots, T\}$. That is, we have
  \begin{align*}
    \lefteqn{\statprob\left(\tmixtv(\statprob_{[s]}, \epsilon)
      > \tmixtv(\statprob, \epsilon) + c
      ~ \mbox{for~some~} s \in \{1, \ldots, T\}
      ~\mbox{and}~ \epsilon \ge \epsilon_T\right)} \\
    & \qquad\qquad\qquad\qquad\qquad\qquad\qquad\qquad
    ~ \le
      \statprob\left(\max_{t, s \le T}
    \left[\tmixtv(\statprob_{[s]}, \epsilon_t) - \tmixtv(\statprob, \epsilon_t)
      \right] > c\right).
  \end{align*}
  Applying a union bound and Assumption~\ref{assumption:probabilistic-mixing},
  we thus see that for any $c \ge 0$,
  \begin{align*}
    \lefteqn{\statprob\left(\max_{s \le T} \sup_{\epsilon \ge \epsilon_T}
    \left(\tmixtv(\statprob_{[s]}, \epsilon) - \tmixtv(\statprob, \epsilon)
    \right) > c\right)} \\
    & \qquad \le T^2 \max_{t, s \le T} P\left(\tmixtv(\statprob_{[s]}, \epsilon_t)
    > \tmixtv(\statprob, \epsilon) + c\right)
    \le T^2 \exp\left(-c / \mixconst\right).
  \end{align*}
  Setting the final equation equal to $\delta$ and solving, we obtain
  $c = \mixconst[\log(1/\delta) + 2 \log(T)]$,
  which is equivalent to the statement of the lemma.
\end{proof-of-lemma}

\comment{
\subsection{A few divergence relationships}

\begin{align*}
  \int|p - q| d\mu
  & = \int_{p > q} (p - q) d\mu
  + \int_{p < q} (q - p) d\mu
  = 1 - 2 \int_{p < q} p d\mu
  + 1 - 2 \int_{q < p} q d\mu \\
  & = 2 - 2 \int p \wedge q d\mu
  \ge 2 - 2 \int \sqrt{pq} d\mu
  = \int(\sqrt{p} - \sqrt{q})^2 d\mu
\end{align*}
So we see that (in our notation)
\begin{equation*}
  \dhel(\statprob, Q)^2 \le \dtv(\statprob, Q).
\end{equation*}
On the other hand,
\begin{align*}
  \int |p - q| d\mu
  & = \int |p / \sqrt{q} - \sqrt{q}| \sqrt{q} d\mu
  \le \left(\int (p / \sqrt{q} - \sqrt{q})^2 d\mu
  \right)^{\half} \left(\int q d\mu\right)^{\half} \\
  & = \left(\int\left(\frac{p}{q} - 1\right)^2 q d\mu \right)^{\half}
  = \sqrt{\dchi(P, Q)}
\end{align*}
and
\begin{align*}
  \int |p - q| d\mu
  & = \int \left|(\sqrt{p} + \sqrt{q})(\sqrt{p} - \sqrt{q})\right| d\mu \\
  & \le \left(\int(\sqrt{p} - \sqrt{q})^2d\mu\right)^{\half}
  \left(\int(\sqrt{p} + \sqrt{q})^2 d\mu\right)^{\half} \\
  & \le \sqrt{2} \dhel(P, Q).
\end{align*}

}

%% file: ergodic_gradient_descent.bbl
\begin{thebibliography}{10}

\bibitem{AgarwalBaRaWa12}
A.~Agarwal, P.~L. Bartlett, P.~Ravikumar, and M.~J. Wainwright.
\newblock Information-theoretic lower bounds on the oracle complexity of convex
  optimization.
\newblock {\em IEEE Transactions on Information Theory}, 58(5):3235--3249, May
  2012.

\bibitem{Azuma67}
K.~Azuma.
\newblock Weighted sums of certain dependent random variables.
\newblock {\em Tohoku Mathematical Journal}, 68:357--367, 1967.

\bibitem{BeckTe03}
A.~Beck and M.~Teboulle.
\newblock Mirror descent and nonlinear projected subgradient methods for convex
  optimization.
\newblock {\em Operations Research Letters}, 31:167--175, 2003.

\bibitem{Ben-TalMaNe01}
A.~Ben-Tal, T.~Margalit, and A.~Nemirovski.
\newblock The ordered subsets mirror descent optimization method with
  applications to tomography.
\newblock {\em SIAM Journal on Optimization}, 12:79--108, 2001.

\bibitem{Bertsekas73}
D.~P. Bertsekas.
\newblock Stochastic optimization problems with nondifferentiable cost
  functionals.
\newblock {\em Journal of Optimization Theory and Applications},
  12(2):218--231, 1973.

\bibitem{Billingsley86}
P.~Billingsley.
\newblock {\em Probability and Measure}.
\newblock Wiley, {S}econd edition, 1986.

\bibitem{BoydGhPrSh06}
S.~Boyd, A.~Ghosh, B.~Prabhakar, and D.~Shah.
\newblock Randomized gossip algorithms.
\newblock {\em IEEE Transactions on Information Theory}, 52(6):2508--2530,
  2006.

\bibitem{BoydVa04}
S.~Boyd and L.~Vandenberghe.
\newblock {\em Convex Optimization}.
\newblock Cambridge University Press, 2004.

\bibitem{Bradley05}
R.~C. Bradley.
\newblock Basic properties of strong mixing conditions. a survey and some open
  questions.
\newblock {\em Probability Surveys}, 2:107--144, 2005.

\bibitem{Chung98}
F.~R.~K. Chung.
\newblock {\em Spectral Graph Theory}.
\newblock AMS, 1998.

\bibitem{CortesVa95}
C.~Cortes and V.~Vapnik.
\newblock Support-vector networks.
\newblock {\em Machine Learning}, 20(3):273--297, September 1995.

\bibitem{CoverTh91}
T.~M. Cover and J.~A. Thomas.
\newblock {\em Elements of Information Theory}.
\newblock Wiley, 1991.

\bibitem{Csiszar67}
I.~Csisz\'ar.
\newblock Information-type measures of difference of probability distributions
  and indirect observation.
\newblock {\em Studia Scientifica Mathematica Hungary}, 2:299--318, 1967.

\bibitem{DuchiAgWa12}
J.~C. Duchi, A.~Agarwal, and M.~J. Wainwright.
\newblock Dual averaging for distributed optimization: convergence analysis and
  network scaling.
\newblock {\em IEEE Transactions on Automatic Control}, 57(3):592--606, 2012.

\bibitem{GelmanRu92}
A.~Gelman and D.~B. Rubin.
\newblock Inference from iterative simulation using multiple sequences.
\newblock {\em Statistical Science}, 7(4):457--472, 1992.

\bibitem{HiriartUrrutyLe96}
J.~Hiriart-Urruty and C.~Lemar\'echal.
\newblock {\em Convex Analysis and Minimization Algorithms I}.
\newblock Springer, 1996.

\bibitem{ImpagliazzoZu89}
R.~Impagliazzo and D.~Zuckerman.
\newblock How to recycle random bits.
\newblock In {\em 30th Annual Symposium on Foundations of Computer Science},
  pages 248--253, 1989.

\bibitem{JarnerRo02}
S.~Jarner and G.~Roberts.
\newblock Polynomial convergence rates of {M}arkov chains.
\newblock {\em The Annals of Applied Probability}, 12(1):pp. 224--247, 2002.

\bibitem{JerrumSi96}
M.~Jerrum and A.~Sinclair.
\newblock The {M}arkov chain {M}onte {C}arlo method: an approach to approximate
  counting and integration.
\newblock In D.~S. Hochbaum, editor, {\em Approximation Algorithms for NP-hard
  Problems}. PWS Publishing, 1996.

\bibitem{JohanssonRaJo09}
B.~Johansson, M.~Rabi, and M.~Johansson.
\newblock A randomized incremental subgradient method for distributed
  optimization in networked systems.
\newblock {\em SIAM Journal on Optimization}, 20(3):1157--1170, 2009.

\bibitem{KarzanovKh91}
A.~Karzanov and L.~Khachiyan.
\newblock On the conductance of order {M}arkov chains.
\newblock {\em Order}, 8:7--15, 1991.

\bibitem{KushnerYi03}
H.~J. Kushner and G.~Yin.
\newblock {\em Stochastic Approximation and Recursive Algorithms and
  Applications}.
\newblock Springer, {S}econd edition, 2003.

\bibitem{LesserOrTa03}
V.~Lesser, C.~Ortiz, and M.~Tambe, editors.
\newblock {\em {Distributed Sensor Networks: A Multiagent Perspective}},
  volume~9.
\newblock Kluwer Academic Publishers, 2003.

\bibitem{Liebscher05}
E.~Liebscher.
\newblock Towards a unified approach for proving geometric ergodicity and
  mixing properties of nonlinear autoregressive processes.
\newblock {\em Journal of Time Series Analysis}, 26(5):669--689, 2005.

\bibitem{Matousek02}
J.~Matousek.
\newblock {\em Lectures on Discrete Geometry}.
\newblock Springer, 2002.

\bibitem{Mesterharm05}
C.~Mesterharm.
\newblock On-line learning with delayed feedback.
\newblock In {\em Algorithmic Learning Theory}, pages 399--413, 2005.

\bibitem{MeynTw09}
S.~Meyn and R.~L. Tweedie.
\newblock {\em Markov Chains and Stochastic Stability}.
\newblock Cambridge University Press, {S}econd edition, 2009.

\bibitem{Mokkadem88}
A.~Mokkadem.
\newblock Mixing properties of {ARMA} processes.
\newblock {\em Stochastic Processes and their Applications}, 29(2):309--315,
  1988.

\bibitem{NedicBe01}
A.~Nedi\'{c} and D.~P. Bertsekas.
\newblock Incremental subgradient methods for nondifferentiable optimization.
\newblock {\em SIAM Journal on Optimization}, 12:109--138, 2001.

\bibitem{NemirovskiJuLaSh09}
A.~Nemirovski, A.~Juditsky, G.~Lan, and A.~Shapiro.
\newblock Robust stochastic approximation approach to stochastic programming.
\newblock {\em SIAM Journal on Optimization}, 19(4):1574--1609, 2009.

\bibitem{NemirovskiYu83}
A.~Nemirovski and D.~Yudin.
\newblock {\em Problem Complexity and Method Efficiency in Optimization}.
\newblock Wiley, 1983.

\bibitem{PolyakJu92}
B.~T. Polyak and A.~B. Juditsky.
\newblock Acceleration of stochastic approximation by averaging.
\newblock {\em SIAM Journal on Control and Optimization}, 30(4):838--855, 1992.

\bibitem{PolyakTs80}
B.~T. Polyak and J.~Tsypkin.
\newblock Robust identification.
\newblock {\em Automatica}, 16:53--63, 1980.

\bibitem{RamNeVe09a}
S.~S. Ram, A.~Nedi\'c, and V.~V. Veeravalli.
\newblock Incremental stochastic subgradient algorithms for convex
  optimization.
\newblock {\em SIAM Journal on Optimization}, 20(2):691--717, 2009.

\bibitem{RobbinsMo51}
H.~Robbins and S.~Monro.
\newblock A stochastic approximation method.
\newblock {\em Annals of Mathematical Statistics}, 22:400--407, 1951.

\bibitem{RobertCa04}
C.~Robert and G.~Casella.
\newblock {\em Monte Carlo Statistical Methods}.
\newblock Springer, {S}econd edition, 2004.

\bibitem{RockafellarWe82}
R.~T. Rockafellar and R.~J.~B. Wets.
\newblock On the interchange of subdifferentiation and conditional expectation
  for convex functionals.
\newblock {\em Stochastics: An International Journal of Probability and
  Stochastic Processes}, 7:173--182, 1982.

\bibitem{Spall03}
J.~C. Spall.
\newblock {\em Introduction to Stochastic Search and Optimization: Estimation,
  Simulation, and Control}.
\newblock Wiley, 2003.

\bibitem{VapnikCh71}
V.~N. Vapnik and A.~Y. Chervonenkis.
\newblock On the uniform convergence of relative frequencies of events to their
  probabilities.
\newblock {\em Theory of Probability and its applications}, XVI(2):264--280,
  1971.

\bibitem{WeiTa90}
G.~Wei and M.~A. Tanner.
\newblock A {M}onte {C}arlo implementation of the {EM} algorithm and the poor
  man's data augmentation algorithms.
\newblock {\em Journal of the American Statistical Association},
  85(411):699--704, 1990.

\bibitem{Wilson04}
D.~B. Wilson.
\newblock Mixing times of lozenge tiling and card shuffling {M}arkov chains.
\newblock {\em Annals of Applied Probability}, 14(1):274--325, 2004.

\bibitem{Yu94}
B.~Yu.
\newblock Rates of convergence for empirical processes of stationary mixing
  sequences.
\newblock {\em Annals of Probability}, 22(1):94--116, 1994.

\bibitem{Zinkevich03}
M.~Zinkevich.
\newblock Online convex programming and generalized infinitesimal gradient
  ascent.
\newblock In {\em Proceedings of the Twentieth International Conference on
  Machine Learning}, 2003.

\end{thebibliography}
